\documentclass[11pt]{amsart}

\usepackage[T1]{fontenc}
\usepackage{lmodern}
\usepackage[english]{babel}
\usepackage{csquotes}

\usepackage{enumitem}
\usepackage{etoolbox}
\robustify\cite
\usepackage{xcolor}
\usepackage{xurl}
\usepackage{ragged2e}
\usepackage{float}

\setlist{nosep}
\setlist[enumerate]{label=(\roman*)}

\usepackage{mathtools}
\mathtoolsset{showonlyrefs=true} 
\usepackage{amssymb} 
\usepackage{bm}
\usepackage[
cal=rsfso,
calscaled=1.05,
scr=rsfso,
scrscaled=1.05
]{mathalpha} 

\usepackage{array}
\usepackage{booktabs}

\usepackage[
backend=biber,
style=alphabetic,      
backref=true,          
maxnames=99,      
minnames=99,
maxalphanames=99, 
minalphanames=99, 
isbn=false,       
url=false,        
doi=true,         
eprint=true       
]{biblatex}

\AtBeginBibliography{\setlength{\emergencystretch}{3em}}

\usepackage[
colorlinks=true,
hypertexnames=false,
linkcolor=red,
citecolor=blue,
filecolor=blue,
urlcolor=blue
]{hyperref}

\usepackage[nameinlink,noabbrev]{cleveref}

\makeatletter

\renewcommand\normalsize{%
 \@setfontsize\normalsize{11.7}{14pt plus .3pt minus .3pt}%
 \abovedisplayskip 10\p@ \@plus4\p@ \@minus4\p@
 \abovedisplayshortskip 6\p@ \@plus2\p@
 \belowdisplayshortskip 6\p@ \@plus2\p@
 \belowdisplayskip \abovedisplayskip
}

\renewcommand\small{%
 \@setfontsize\small{9.5}{12\p@ plus .2\p@ minus .2\p@}%
 \abovedisplayskip 8.5\p@ \@plus4\p@ \@minus1\p@
 \belowdisplayskip \abovedisplayskip
 \abovedisplayshortskip \abovedisplayskip
 \belowdisplayshortskip \abovedisplayskip
}

\renewcommand\footnotesize{%
 \@setfontsize\footnotesize{8.5}{9.25\p@ plus .1pt minus .1pt}%
 \abovedisplayskip 6\p@ \@plus4\p@ \@minus1\p@
 \belowdisplayskip \abovedisplayskip
 \abovedisplayshortskip \abovedisplayskip
 \belowdisplayshortskip \abovedisplayskip
}

\ifdefined\pdfpagewidth
\else
\fi

\calclayout
\makeatother

\newtheorem{theorem}{Theorem}[section]
\newtheorem{corollary}[theorem]{Corollary}
\newtheorem{lemma}[theorem]{Lemma}
\newtheorem{proposition}[theorem]{Proposition}

\theoremstyle{definition}
\newtheorem{definition}[theorem]{Definition}
\newtheorem{example}[theorem]{Example}

\theoremstyle{remark}
\newtheorem{remark}[theorem]{Remark}

\newcommand{\exend}{\hfill$\diamond$}
\AtEndEnvironment{example}{\exend}

\newcommand{\Q}{\mathbb{Q}}
\newcommand{\Z}{\mathbb{Z}}
\newcommand{\C}{\mathbb{C}}
\newcommand{\PP}{\mathbb{P}}

\DeclareMathOperator{\id}{id}
\DeclareMathOperator{\rk}{rk}
\DeclareMathOperator{\pr}{pr}
\DeclareMathOperator{\End}{End}
\DeclareMathOperator{\Bl}{Bl}

\DeclareMathOperator{\Hdg}{Hdg}
\DeclareMathOperator{\Exp}{Exp}
\DeclareMathOperator{\spec}{spec}
\DeclareMathOperator{\mult}{mult}

\newcommand{\rH}{\mathrm{H}}

\newcommand{\Hod}{\mathsf{Hod}}
\newcommand{\Eu}{\mathsf{Eu}}
\newcommand{\Grad}{\mathsf{Gr}}
\newcommand{\NS}{\mathsf{NS}}

\newcommand{\cl}[1]{\overline{#1}}
\newcommand{\ms}[1]{\{\!\{#1\}\!\}}

\newcommand{\cE}{\mathcal{E}}
\newcommand{\cH}{\mathcal{H}}
\newcommand{\cL}{\mathcal{L}}
\newcommand{\cM}{\mathcal{M}}
\newcommand{\cN}{\mathcal{N}}
\newcommand{\cO}{\mathcal{O}}
\newcommand{\cR}{\mathcal{R}}
\newcommand{\cV}{\mathcal{V}}

\begin{document}
 
 \title{A quantum connection approach to rationality}
 
 \author[Carazolli]{Luiz Carazolli}
 \address{Instituto de Matemática, Estatística e Computação Científica (IMECC) da Universidade Estadual de Campinas (Unicamp), Cidade Universitária, Campinas - SP, 13083-856, Brazil.}
 \email{luiz.otavio.carazolli@gmail.com}
 
 \author[Cavenaghi]{Leonardo F. Cavenaghi}
 \address{Institute of Mathematics and Informatics, Bulgarian Academy of Sciences}
 \email{leonardofcavenaghi@gmail.com}
 
 \author[Katzarkov]{Ludmil Katzarkov}
 \address{Simons Center for Geometry and Physics, Stony Brook University, Stony Brook, NY 11794, USA. \& International Center for Mathematical Sciences (ICMS), Sofia, Bulgaria. \& Institute of Mathematics and Informatics, Bulgarian Academy of Sciences, Sofia, Bulgaria.}
 \email{lkatzarkov@gmail.com}
 
 \author[Kontsevich]{Maxim Kontsevich}
 \address{Institut des Hautes Études Scientifiques, 35 route de Chartres, F - 91440 Bures-sur-Yvette}
 \email{maxim@ihes.fr}
 
 \author[Martins]{Pedro Antonio Muniz Martins}
 \address{Instituto de Matemática, Estatística e Computação Científica (IMECC) da Universidade Estadual de Campinas (Unicamp), Cidade Universitária, Campinas - SP, 13083-856, Brazil.}
 \email{pedroa.muniz9@gmail.com}
 
 \begin{abstract}
  We extract rank and regularized exponents from cyclic generic factors
  of the quantum connection and show that these data persist under
  the birational comparisons considered here. Applications include showing non-rationality of
  cubic threefolds, ordinary Gushel--Mukai threefolds, very general
  quartic fourfolds, and products of the threefolds with $\PP^1$.
 \end{abstract}
 
 \maketitle
 
 \section{Introduction}
 \label{sec:introduction}
 
 Let $Y\subset\PP^n$ be a smooth complex Fano complete intersection of
 multidegree $(d_1,\ldots,d_r)$ and dimension $m=n-r\geq3$,
 where $2\leq d_1\leq\cdots\leq d_r$.
 By the Lefschetz hyperplane theorem, $\operatorname{Pic}(Y)=\Z\cdot[\mathcal O_Y(1)]$, and adjunction gives
 \[
 \omega_Y^{-1}\cong\mathcal O_Y(r_Y),
 \qquad
 r_Y=n+1-\sum_{i=1}^r d_i.
 \]
 The integer $r_Y>0$ is the Fano index.
 
 Every smooth complete intersection over $\C$ is smoothly attainable
 in the sense of \cite[Def.~1.5 and Lem.~6.11]
 {kuznetsov2021serrefunctorsdimensionsresidual}. The bounded derived category of coherent sheaves of $Y$ admits the
 semiorthogonal decomposition
 \[
 \mathrm D^{\mathrm b}(Y)
 =
 \left\langle
 \mathcal A_Y,\mathcal O_Y,\mathcal O_Y(1),
 \ldots,\mathcal O_Y(r_Y-1)
 \right\rangle,
 \]
 and 
 \cite[Thm.~1.7]{kuznetsov2021serrefunctorsdimensionsresidual} identifies the upper Serre dimension of the residual category $\mathcal A_Y$ as
 \begin{equation}\label{eq:intro-serre-dimension}
  \sigma(Y):=m-\frac{2r_Y}{d_r}.
 \end{equation}
 This paper is in part motivated by the following question:
 \begin{quote}
  \textbf{Question.}
  If
  \[
  \sigma(Y)>m-2,
  \qquad\text{equivalently}\qquad
  r_Y<d_r,
  \]
  is a very general smooth Fano complete intersection of the specified
  multidegree irrational? In which cases does irrationality hold for
  every smooth member?
 \end{quote}
 
 This question is motivated by the expected relation between Hodge atoms \cite{KKPY} and admissible components of derived categories; see \cite[\S2.7]{cavenaghi2026atomsmeetsymbols}. The non-rationality criterion in the theory \cite[Prop.~5.30]{KKPY} states that every atom of a rational smooth projective $m$-fold, with $m\geq2$, occurs on a smooth projective variety of dimension at most $m-2$. Our motivating question is an instance of the following: does there exist a numerical invariant of individual atoms that takes the value $\sigma(Y)$ on an atom of $Y$ and is bounded above by the dimension of every variety on which that atom occurs?
 
 \
 
 For a smooth complex projective variety $X$, genus-zero Gromov--Witten invariants define a family of products $\star$ on its rational Betti cohomology. Over a complete non-Archimedean field $\Bbbk$, one can build a parameter space (a $\Bbbk$-super analytic manifold, in the sense of Berkovich) $B_X$, such that under the tangent-space identification with cohomology, each fiber acquires the structure of a Frobenius superalgebra
 \[
 T_bB_X\cong\left(\rH^\ast(X,\Q)\otimes_\Q\Bbbk,\star_b\right).
 \]
 Gromov--Witten invariants are deformation invariants. The theory of Hodge Atoms \cite{KKPY} relates them to birational geometry by understanding how quantum multiplication interacts with the Hodge structure. More precisely, in Section~\ref{sec:hodge} we recall that there exists an algebraic reductive group $\mathsf{Hod}(X)$ obtained as the image of the universal Hodge group $\mathsf{Hod}$ on $\rH^*(X,\Q)$ whose analytic group germ  $\mathsf{Hod}(X)_\Bbbk$ at the identity acts on $B_X$. We denote its fixed locus by $B_X^{\Hod}$. Let $U_X^{\mathrm{gen}}\subset B_X^{\Hod}$ be the locus where the operator $\Eu_b\star_b(-)$ has the maximal number of distinct eigenvalues, where $\Eu=c_1(T_X){+}\ldots$ is a distinguished even vector field named \emph{Euler vector field} (see Equation~\eqref{eq:euler-grading}). Pick a $\Bbbk$-rational point $b\in U_X^{\mathrm{gen}}(\Bbbk)$ and write $K_b^{\mathrm F}=-\Eu_b\star_b$. Set
 \[
 H_{\lambda,b}:=\ker(K_b^{\mathrm F}-\lambda\id)^N,
 \qquad N:=\dim_\Q\rH^\ast(X,\Q).
 \]
 Let $V_j$ denote the isotypic components of $\rH^\ast(X,\Q)\otimes_\Q\Bbbk$ as a representation of
 $\mathsf{Hod}(X)_{\Bbbk}$. Then the action of $K_b^{\mathrm F}$ on the tangent bundle commutes with this action, hence preserves every $V_j$, and
 \begin{equation}\label{eq:intro-compatible-decompositions}
  T_bB_X
  =
  \bigoplus_{\lambda\in\spec(K_b^{\mathrm F})}H_{\lambda,b}
  =
  \bigoplus_{\lambda\in\spec(K_b^{\mathrm F})}
  \bigoplus_j\bigl(V_j\cap H_{\lambda,b}\bigr).
 \end{equation}
 Thus, each generalized $\lambda$-eigenspace for $K_b^{\mathrm F}$ can be thought of as a representation of $\Hod(X)$ whose isomorphism class (termed a \emph{coarse Hodge atom} in \cite{CKK-lectures,alexeev2026atomicdecompositioncompleteintersection})  is independent of the point along a connected component of the reduced generic spectral cover; see \cite[Prop.~5.23]{KKPY}. The theory of Hodge atoms is built using the notion of \emph{F-bundles}, which we recall in Section~\ref{sec:F-bundles}. 

 Let $\mathbb D$ be the germ of an analytic disk at $u=0$. Then, there is a supervector bundle $\mathcal H$ over $B_X\times\mathbb D$ equipped with a meromorphic flat connection $\nabla$ with poles at $u=0$ and such that
 \[K^{\mathrm F}_b=\nabla_{u^2\partial_u}|_{\mathcal H_{b,0}},~\nabla_{u\xi}|_{\mathcal H_{b,0}}\]
 are regular operators for any local vector field $\xi$. 
 
 In this paper, we consider the action of $K^{\mathrm{F}}_b$ on the Hodge-invariant part of the cohomology of $X$. Namely, set
 \[
 H_{\Hdg}(X):=\bigoplus_p
 \bigl(\rH^{2p}(X,\Q)\cap H^{p,p}(X)\bigr),
 \qquad
 H^{\Hdg}_{\lambda,b}
 :=H_{\lambda,b}\cap
 \bigl(H_{\Hdg}(X)\otimes_\Q\Bbbk\bigr).
 \]
 As explained in Section~\ref{sec:hodge}, the connection restricts to a purely even maximal F-bundle
 $\cH_{\Hdg}$ over $B_X^{\Hod}$. Its leading endomorphism is
 $K_b^{\mathrm F}|_{H_{\Hdg}(X)\otimes\Bbbk}$.
 Theorem~\ref{thm:hyzz} gives, on suitable germs,
 \begin{equation}\label{eq:intro-local-f-bundles}
  (B_X^{\Hod},b)\cong
  \prod_{\lambda\in\spec(K_b^{\mathrm F})}(B_\lambda,b_\lambda),
  \qquad
  (\cH_{\Hdg},\nabla)\cong
  \boxplus_{\lambda\in\spec(K_b^{\mathrm F})}
  (\cH^{\Hdg}_\lambda,\nabla^{\Hdg}_\lambda),
 \end{equation}
 where
 $(\cH^{\Hdg}_\lambda)_{b_\lambda,0}
 \cong H^{\Hdg}_{\lambda,b}$.
 Here, $\boxplus$ denotes the external sum, which we recall in
 Section~\ref{sec:quantum}. The second isomorphism in
 \eqref{eq:intro-local-f-bundles} is an isomorphism of F-bundles,
 including their connections in the disk direction. At a fixed
 point $b$, restrict to the disk with coordinate $u$, complete at
 $u=0$, and invert $u$. On the Hodge-invariant fiber, this gives the formal
 meromorphic disk connection
 \[
 \nabla_{\partial_u,b}
 =\partial_u+\frac{K_b^{\mathrm F}}{u^2}+\frac{\Grad}{u},
 \]
 with $K_b^{\mathrm F}$ and the \emph{grading operator} $\Grad$ (explicitly defined in Equation~\eqref{eq:euler-grading}) restricted to $H_{\Hdg}(X)\otimes\Bbbk$.
 The F-bundle decomposition then gives differential summands $\cM_{\lambda,b}$ of this connection, named \emph{generic Hodge factors}, indexed by the eigenvalues
 $\lambda$ of $K_b^{\mathrm F}$.
 
 We say that $\cM_{\lambda,b}$  is \emph{cyclic} if the restriction of $K_b^{\mathrm F}$ to the corresponding generalized eigenspace has a single Jordan block. We propose birational invariants when the differential modules $\cM_{\lambda,b}$ are of \emph{unramified exponential type}  (Definition~\ref{def:exponents}):
 \[
 \cM_{\lambda,b}
 \cong\cE^{-\lambda/u}\otimes\cR_{\lambda,b},
 \]
 where $\cR_{\lambda,b}$ is \emph{regular singular}. Our invariants are the ranks of these modules and the multisets $F:=\Exp(\cR_{\lambda,b})$ of negative residue eigenvalues modulo $\Z$, counted with algebraic multiplicity. We remark that these are not invariants of Hodge atoms, however. Here we prove (Theorem~\ref{thm:criterion}):  let $X$ be a connected smooth complex projective variety of dimension $2\leq n\leq4$. Suppose that a generic Hodge factor  has a cyclic representative of unramified exponential type
  with rational residues. Let $m$ be its rank and $F$ its regularized exponent multiset. If $X$ is rational, then $(m,F)$ are as one of\footnote{Throughout, $\ms{a_1,\ldots,a_k}$ denotes a multiset, and $\cl a:=a+\Z$ for $a\in\Q$.}:
 \begin{equation}\label{eq:intro-allowed-types}
  \begin{array}{c|c|c}
   \text{type} & \text{rank~$m$} & \text{exponent multiset~$F$}\\
   \hline
   \text{point}   & 1 & \ms{\cl0}\\
   \text{curve}   & 2 & \ms{\cl{1/2},\cl{1/2}}\\
   \text{surface} & 3 & \ms{\cl0,\cl0,\cl0}
  \end{array}
 \end{equation}

To prove Theorem~\ref{thm:criterion}, we check that if a connected component of the generic spectral cover has one $\Bbbk$-rational cyclic point at which the selected factor is of unramified exponential type with rational exponents, then every $\Bbbk$-rational cyclic point on that component has the same properties and exponent multiset (Proposition~\ref{prop:cyclic-transport}). Moreover, we show that a selected factor satisfying these hypotheses can also be followed through the blowup and projective-bundle comparisons while preserving its rank, cyclicity, and exponents (Lemma~\ref{lem:separation}). These results depend on the formal-persistence statement given in \cite[Rem.~3.14]{KKPY}, proven in Theorem~\ref{cl:formal-persistence}.

    \
 
Section~\ref{sec:examples} gives the following applications.
 
 For cubic threefolds, we recover the irrationality theorem of Clemens--Griffiths \cite{CG}. We remark that Cai also obtains an obstruction from the formal monodromy of the
 quantum connection \cite{cai2026cubicthreefoldsymplecticallyirrational}. Moreover, we prove that $Y\times\PP^1$ is irrational for every smooth cubic threefold $Y$, a conclusion also established by Gong--Xu \cite[Thm.~1.1]{gong2026irrationalitystabilizationcomplexsmooth}. For Gushel--Mukai threefolds, Beauville proved irrationality of a
 general member \cite[Thm.~5.6(ii)]{Beauville1977}, and
 Debarre--Mongardi constructed explicit irrational examples \cite{DebarreMongardi}. Our criterion proves the irrationality of every smooth ordinary Gushel--Mukai threefold and of its product with $\PP^1$. For very general quartic fourfolds, Totaro proved stable irrationality \cite{Totaro2015}. Our criterion recovers the irrationality conclusion.\footnote{During the preparation of this manuscript, the authors learned that Andrew Harder, Jiaji Cai, and Sukjoo Lee were independently working out this example via similar methods.} The authors remark that a catalog computing many small quantum multiplication matrices was built in \url{https://leonardofcavenaghi.github.io/atomicdecompositions/}. A software package is also presented for computing these matrices and other genus-0 Gromov--Witten invariants. The authors remark that this website (and software) will be updated to compute the exponent multiset for many smooth projective 3-folds, yielding other applications. For instance, one anticipates that the smooth complete intersection of a quadric and a cubic in $\mathbb P^5$ can be shown to be irrational via Theorem~\ref{thm:criterion} and this mentioned computation, offering an alternative to \cite{Beauville1977, Beauville2012, Pukhlikov1989Maximal}.
 
 \
 
Returning to our starting question, we expect that the invariant $\sigma(Y)$ is related to the exponents $F$ of a given Hodge factor. More precisely, in the case of smooth cubic threefolds (Example~\ref{ex:cubic3}), we check for the eigenvalue $0$, that the representatives $\bmod \mathbb Z$ are $s=-\dfrac{5}{6},~-\dfrac{1}{6}$. The number $-\frac{5}{6}$ matches $-\dfrac{1}{2}\times $ the upper Serre dimension for the Kuznetsev component of the SOD decomposition in \cite{kuznetsov2021serrefunctorsdimensionsresidual}. This holds likewise for smooth quartic fourfolds here computed (Example~\ref{ex:quartic4}). We check that the same also holds for Gushel--Mukai threefolds (Example~\ref{ex:GM3}) for the eigenvalue $6q$, where $q$ is a Novikov variable. The corresponding upper Serre dimension is computed in~\cite{perry2022stability}.

The relation  between conformal field theory, asymptotics of quantum
 differential equations, and Steenbrink spectra, initiated by  Cecotti--Vafa \cite{CECOTTI1991359} and  Katzarkov--Lee--Svoboda--Petkov \cite{KLSP}, remains a possible  source of stronger obstructions to rationality \cite{CKK-lectures}.

 \medskip
 
 \subsection*{Acknowledgements}
Authors are in debt with D.~Auroux, A.~Harder, and T.~Pantev for useful discussions and comments throughout the development of this work.
 
 L.~Carazolli is supported by the Coordenação de Aperfeiçoamento de Pessoal de Nível Superior – Brasil (CAPES) – Finance Code 001.
 
 L.~F.~Cavenaghi and L.~Katzarkov are supported by the Simons Foundation, grant SFI-MPS-T-Institutes-00007697, and the Ministry of Education and Science of the Republic of Bulgaria, grant DO1-239/10.12.2024. 
 
 L.~Katzarkov is supported by the Simons Investigators Award (no.003136), NSF FRG grant DMS-2245099, and NSF FRG grant DMS-2255039.
 
 P.~A.~M.~Martins is supported by FAPESP grants no. 2024/07684-7, 2025/22312-1. 
 
 In addition, Ludmil Katzarkov and Maxim Kontsevich are grateful to Jeffrey Fuqua for his continuous financial support of mathematical research throughout the years.
 
 \medskip
 
 \section{Basics on F-bundles}
 \label{sec:quantum}
 In this section, we follow \cite{HYZZ, KKPY} to recall the basics of F-bundles and the necessary structural results used throughout. Let $t$ denote a formal auxiliary variable. We work over the non-Archimedean field $\Bbbk$, which is the completion of the Puiseux series field:
 \[
 \Bbbk_{\mathrm{Puis}}:=\bigcup_{j\geq1}\overline{\Q}((t^{1/j})),
 \qquad \Bbbk:=\widehat{\Bbbk_{\mathrm{Puis}}}.
 \]
 The valuation $v(f)$ on $\Bbbk_{\mathrm{Puis}}$ is the least exponent of a nonzero series, with $v(0)=+\infty$. We define the absolute value as $|f|=e^{-v(f)}$, which extends to the completion. The field $\Bbbk$ has characteristic zero and is complete and algebraically closed. Its nontrivial valuation restricts to the trivial valuation on $\overline\Q$. All analytic spaces considered in this paper are $\Bbbk$-analytic spaces in the sense of Berkovich \cite{Berkovich,jonsson_berkovich_2016}. We sometimes also use $\Bbbk$-super analytic spaces: locally, the structure sheaf is the tensor product of an ordinary analytic structure sheaf with a finite exterior algebra in odd generators. A rational point means a $\Bbbk$-rational point. We write $\mathbb D$ for the germ at $u=0$ of the analytic closed unit disk with coordinate $u$.

 \subsection{F-bundles}
 \label{sec:F-bundles}
 
 Let $B$ be a smooth $\Bbbk$-analytic space, or a smooth $\Bbbk$-analytic superspace. An \emph{F-bundle} $(\cH,\nabla)/B$ is a vector bundle (or a super vector bundle) on
 $B\times\mathbb D$, equipped with a flat meromorphic connection having poles only along $u=0$, such that $\nabla_{u^2\partial_u}$ and $\nabla_{u\xi}$ are regular for every analytic vector field $\xi$ on $B$. In the super setting, the connection is even: for a homogeneous vector field $\xi$, the operator $\nabla_\xi$ has parity $|\xi|$. Flatness is expressed using supercommutators.
 Put $\cH_0:=\cH|_{u=0}$. Reduction at $u=0$ gives
 \begin{equation}\label{eq:F-residual-action}
  \mu\colon T_B\longrightarrow\End_{\cO_B}(\cH_0),
  \qquad \mu(\xi):=\nabla_{u\xi}|_{u=0},
  \qquad K^{\mathrm F}:=\nabla_{u^2\partial_u}|_{u=0}.
 \end{equation}
 These are $\cO_B$-linear: the derivative terms in the Leibniz rule
 vanish modulo $u$. We call $K^{\mathrm F}$ the \emph{residual
  endomorphism} and write $K_b^{\mathrm F}$ for its value at $b$.
 Flatness, reduced modulo $u$, gives
 \begin{equation}\label{eq:F-commuting-action}
  [\mu(\xi),\mu(\eta)]_{\mathrm s}=0,
  \qquad [K^{\mathrm F},\mu(\xi)]=0,
 \end{equation}
 where $[G,H]_{\mathrm s}=GH-(-1)^{|G||H|}HG$ for homogeneous
 endomorphisms.
 
 The F-bundle is \emph{maximal} at $b\in B(\Bbbk)$
 \cite[Def.~2.6]{HYZZ} if an even vector $\zeta\in\cH_{b,0}$ makes
 \[
 T_bB\longrightarrow\cH_{b,0},
 \qquad
 \xi\longmapsto\mu_b(\xi)\zeta,
 \]
 an isomorphism. Extending $\zeta$ locally and shrinking $B$ makes this evaluation map an isomorphism of bundles.
 We explain why maximality induces a canonical multiplication on
 $T_B$ \cite[Lem.~3.9 and Def.~3.11]{KKPY}.
 If a homogeneous endomorphism $A$ supercommutes with every
 $\mu(\xi)$, choose the unique vector field $\eta$ with
 $A\zeta=\mu(\eta)\zeta$. Since $\zeta$ is even, $|\eta|=|A|$, and
 \[
 A\mu(\xi)\zeta
 =(-1)^{|A||\xi|}\mu(\xi)A\zeta
 =\mu(\eta)\mu(\xi)\zeta.
 \]
 The vectors $\mu(\xi)\zeta$ span $\cH_0$, so $A=\mu(\eta)$.
 Thus, the image of $\mu$ equals its supercentralizer. In particular,
 it contains the identity, all products $\mu(\xi)\mu(\eta)$, and
 $K^{\mathrm F}$. The formulas
 \begin{equation}\label{eq:F-tangent-product}
  \mu(\xi\circ\eta)=\mu(\xi)\mu(\eta),
  \qquad \mu(e)=\id,
  \qquad \mu(\Eu^{\mathrm F})=K^{\mathrm F}
 \end{equation}
 define an associative, supercommutative tangent product,
 its even unit $e$, and an even vector field $\Eu^{\mathrm F}$.
 They are analytic and independent of $\zeta$, since their defining
 identities use only the injective map $\mu$. The evaluation map identifies $\cH_0$ with the
 regular representation of this tangent algebra.
 We call $\Eu^{\mathrm F}$ the \emph{F-bundle Euler field}
 \cite[Def.~3.12]{KKPY}. In particular, with
 $(\Eu^{\mathrm F})^{\circ0}:=e$, every polynomial $p\in\Bbbk[z]$ satisfies
 \begin{equation}\label{eq:F-Euler-powers}
  \mu\bigl(p_\circ(\Eu^{\mathrm F})\bigr)=p(K^{\mathrm F}).
 \end{equation}
 This equation will be largely used when proving the results in Appendix~\ref{app:formal-persistence}.
 
 Given F-bundles $(\cH_i,\nabla_i)/B_i$, their
 \emph{external sum} \cite[Def.~2.11]{HYZZ}, \cite[Def.~3.3]{KKPY} is the
 F-bundle over the product $\prod_iB_i$, where $\pr_i\colon(\prod_jB_j)\times\mathbb D\to B_i\times\mathbb D$ denotes the projection,
 \[
 \boxplus_i(\cH_i,\nabla_i)
 :=\Bigl(\bigoplus_i\pr_i^*\cH_i,\ \bigoplus_i\pr_i^*\nabla_i\Bigr).
 \]
 
 We require the following decomposition theorem:
 
 \begin{theorem}[{\cite[Thm.~1.2, Thm.~3.42]{HYZZ}}]\label{thm:hyzz}
  Let $B$ be a smooth, purely even $\Bbbk$-analytic space, let $(\cH,\nabla)/B$ be maximal at $b\in B(\Bbbk)$, and let
  $\cH_{b,0}=\bigoplus_iH_i$ be a $K_b^{\mathrm F}$-stable decomposition with
  pairwise disjoint spectra. Then, on an admissible neighborhood
  $U\cong\prod_iU_i$ of $b$, there are maximal F-bundles
  $(\cH_i,\nabla_i)/U_i$ with $(\cH,\nabla)|_U\cong\boxplus_i(\cH_i,\nabla_i)$,
  inducing the prescribed decomposition of $\cH_{b,0}$.
 \end{theorem}
 
 The main example of a maximal F-bundle of interest in this paper is the \emph{A-model F-bundle}, which we describe next.
 
 \subsection{A-model F-bundle}
 \label{sec:A-model}
 Let $X$ be a connected smooth complex projective variety of complex dimension $n$. We write $\rH^*(X,\Q)$ for its singular cohomology, graded by cohomological degree, and set $\rH^*(X, F):=\rH^*(X,\Q)\otimes_\Q F$ for any field extension $F/\Q$. For a homogeneous class $\gamma$, $\deg\gamma$ denotes its cohomological degree.
 We write $\NS(X):=\operatorname{Pic}(X)/\operatorname{Pic}^0(X)$ for its N\'eron--Severi group with rank
 $\rho:=\operatorname{rank}\NS(X)$. Choose a homogeneous basis $\{T_i\}$ of $\rH^*(X,\Q)$ such that:
 \begin{itemize}
  \item[(a)] $T_0=\mathbf1$;
  \item[(b)] $T_1,\ldots, T_\rho$ are cohomology classes of an
  integral basis of $\NS(X)/\mathrm{torsion}$;
  \item[(c)] we extend $T_0,\ldots,T_\rho$ to a homogeneous basis of $H_{\Hdg}(X)$ and then choose a homogeneous basis of a $\mathsf{Hod}(X)$-stable complement in $\rH^{\mathrm{ev}}(X,\Q)$, as explained in Section~\ref{sec:hodge};
  \item[(d)] we choose the odd classes freely.
 \end{itemize}
 The coordinates inherit the parity of their corresponding classes, and we use the Koszul sign rule throughout. We write the divisor coordinates multiplicatively as $q_1,\ldots,q_\rho$ and the remaining coordinates as $y_i$. For an effective curve class $\beta$, we set 
 \[
 q^\beta:=\prod_{m=1}^{\rho}q_m^{\int_\beta T_m}.
 \] 
 The \emph{torus domain} is defined as
 \[
 U_X:=\left\{q\in(\mathbb G_{m,\Bbbk}^{\rho})^{\mathrm{an}}:
 -\sum_m\log|q_m|\,T_m\text{ is ample in }\NS(X)\otimes\mathbb R\right\}.
 \]
 The \emph{A-model base} $B_X$ of \cite[\S3.5.2]{KKPY} is the product of $U_X$, the
 analytic affine line in the $y_0$-direction, the open unit polydisk in the
 remaining even coordinates, and the purely odd superspace of the odd
 coordinates. Its tangent sheaf is trivialized by
 $q_m\partial_{q_m}\mapsto T_m$, $\partial_{y_i}\mapsto T_i$, and by
 \cite[Lem.~3.29]{KKPY} the \emph{big quantum product}
 \begin{equation}\label{eq:gw-product}
  g(\gamma_1\star_{\tau}\gamma_2,\gamma_3)
  =\sum_{\beta}\sum_{k\geq0}\frac{q^\beta}{k!}
  \langle\gamma_1,\gamma_2,\gamma_3,\tau,\dots,\tau\rangle_{0,k+3,\beta},
  \qquad
  g(\gamma_1,\gamma_2)=\int_X\gamma_1\cup\gamma_2,
 \end{equation}
 is analytic on $B_X$. Here $\beta$ ranges over effective curve classes, including zero, and $g$ is the Poincar\'e pairing. The insertion $\tau$ is
 \[
 \tau=\sum_{i\notin\{1,\ldots,\rho\}}y_iT_i.
 \]
 We use the primary Gromov--Witten invariants
 \[
 \langle\gamma_1,\ldots,\gamma_r\rangle_{0,r,\beta}
 :=\int_{[\overline{\mathcal M}_{0,r}(X,\beta)]^{\mathrm{vir}}}
 \prod_{j=1}^r\operatorname{ev}_j^*\gamma_j,
 \]
 where $\overline{\mathcal M}_{0,r}(X,\beta)$ is the moduli stack of
 genus-zero stable maps and $\operatorname{ev}_j$ is evaluation at the
 $j$-th marked point.
 
 The \emph{A-model F-bundle} has an underlying trivial super vector bundle
 $\cH:=\rH^*(X,\Q)\otimes_\Q\cO_{B_X\times\mathbb D}$ and
 \emph{quantum connection}
 \begin{equation}\label{eq:quantum-connection}
  \begin{gathered}
   \nabla_{\partial_u}=\partial_u-\frac{1}{u^2}(\Eu\star-)+\frac1u\Grad,\\
   \nabla_{q_m\partial_{q_m}}=q_m\partial_{q_m}+\frac1u(T_m\star-),
   \qquad
   \nabla_{\partial_{y_i}}=\partial_{y_i}+\frac1u(T_i\star-),
  \end{gathered}
 \end{equation}
 where
 \begin{equation}\label{eq:euler-grading}
  \Grad|_{\rH^a(X)}=\frac{a-n}{2}\,\id,
  \qquad
  \Eu=c_1(TX)+\sum_{i\notin\{1,\dots,\rho\}}
  \Bigl(1-\frac{\deg T_i}{2}\Bigr)y_iT_i.
 \end{equation}
 Write $G:=\Grad$. The string equation (or unit axiom) gives
 $\partial_{y_0}(\star)=0$ and $\partial_{y_0}\Eu=\mathbf1$.
 Under the tangent trivialization above, denote by $T_\xi$ the cohomology
 section corresponding to a vector field $\xi$. Then
 $\mu(\xi)=T_\xi\star(-)$ and $\mu(\xi)\mathbf1=T_\xi$, proving
 maximality at every rational point. Moreover,
 $T_{\xi\circ\eta}=T_\xi\star T_\eta$ and $T_e=\mathbf1$:
 the product induced in \eqref{eq:F-tangent-product} is the
 quantum product under this trivialization. We have
 \begin{equation}\label{eq:sign}
  K^{\mathrm F}=-\Eu\star(-),
  \qquad \Eu^{\mathrm F}=-\Eu,
  \qquad
  \nabla_{\partial_u}=\partial_u+\frac{K^{\mathrm F}}{u^2}+\frac Gu
 \end{equation}
 where we identify $\Eu$ with its tangent vector field. Thus, $\Eu$ always denotes the ``quantum Euler field'' of
 \eqref{eq:euler-grading}, whereas $\Eu^{\mathrm F}$ denotes the
 F-bundle Euler field. Their powers span the same distribution.
 Throughout the paper, spectral eigenvalues refer to $K^{\mathrm F}$.

 \subsection{Rational Hodge classes}\label{sec:hodge}
 \subsubsection*{Basics on Hodge theory}
 Write
 \[
 H_{\Hdg}(X):=
 \bigoplus_{p=0}^{n}
 \bigl(\rH^{2p}(X,\Q)\cap H^{p,p}(X)\bigr),
 \]
 where the intersection is taken inside $\rH^{2p}(X,\C)$ and $H^{p,p}(X)$ denotes the corresponding Hodge summand.
 Recall (see, e.g., \cite{moonen1999notes}) that the category $\mathbb{Q}\mathbf{HS}$, whose objects are finite direct sums of pure polarizable $\mathbb{Q}$-Hodge structures, is Tannakian. Consequently, it is equivalent to the representation category $\operatorname{Rep}\mathsf{MT}$, where $\mathsf{MT}$ is the pro-reductive algebraic group over $\mathbb{Q}$ known as the \emph{universal Mumford-Tate group}. Concretely, recall that a pure $\mathbb{Q}$-Hodge structure of weight $n\in \mathbb{Z}$ consists of a finite-dimensional $\mathbb{Q}$-vector space $V$ and a homomorphism $h\colon\mathbb{S}\rightarrow \mathrm{GL}(V)_{\mathbb{R}}$ of algebraic groups over $\mathbb{R}$, where $\mathbb S:=\operatorname{Res}_{\C/\mathbb R}\mathbb G_{m,\C}$ is the \emph{Deligne torus} \cite[\S\S1.2--1.3]{moonen1999notes}, such that: $h\circ w\colon\mathbb{G}_{m,\mathbb{R}}\rightarrow\mathrm{GL}(V)_{\mathbb{R}}$ is given by $z\mapsto z^{-n}\cdot \mathrm{id}_V$, where $w\colon \mathbb{G}_{m,\mathbb{R}}\rightarrow \mathbb{S}$ is the weight cocharacter given on points by the natural inclusion $\mathbb{R}^{\times}=\mathbb{G}_{m,\mathbb{R}}(\mathbb{R})\hookrightarrow \mathbb{S}(\mathbb{R})\cong \mathbb{C}^{\times}$. A point $(z_1,z_2)\in \mathbb{S}(\mathbb{C})\cong \mathbb{C}^{\times}\times\mathbb{C}^{\times}$ acts on $V^{p,q}$ as $z_1^{-p}z_2^{-q}$, where
 \[
 V\otimes_{\mathbb{Q}}\mathbb{C}=\bigoplus_{p+q=n}V^{p,q}.
 \]
 The \emph{Tate Hodge structure} $\Q(-n)$ has underlying rational vector space $(2\pi\sqrt{-1})^{-n}\Q$ and is pure of type $(n,n)$, hence of weight $2n$. A polarization of $V$ is a morphism of Hodge structures \[   \varphi\colon V\otimes V\longrightarrow\Q(-n) \] such that the rational bilinear form $S:=(2\pi\sqrt{-1})^n\varphi$ is $(-1)^n$-symmetric and the Hermitian form \[   (v,w)\longmapsto S(Cv,\overline w) \] on $V_\C$ is positive definite, where the Weil operator $C$ acts on $V^{p,q}$ as $(\sqrt{-1})^{p-q}$.
 
 The Mumford-Tate group $\mathsf{MT}(V)$ of a specific Hodge structure $V$ is the smallest algebraic subgroup $M \subset \mathrm{GL}(V)$ defined over $\mathbb{Q}$ such that $h$ factors through $M_{\mathbb{R}}$. The category $\Q\mathbf{HS}$ is neutral Tannakian with respect to the forgetful fiber functor $\omega$. If $\langle V\rangle^\otimes\subset\Q\mathbf{HS}$ denotes the full Tannakian subcategory generated by the Hodge structure $V$, then \[   \operatorname{Aut}^{\otimes}   \bigl(\omega|_{\langle V\rangle^\otimes}\bigr)   \cong\mathsf{MT}(V). \]
 
 Let $\chi_{\mathrm{Tate}}\colon\mathsf{MT}\longrightarrow   \mathbb G_m$ be the character describing the action on $\Q(1)$. We define the \emph{universal Hodge group} by \[   \mathsf{Hod}:=\ker\chi_{\mathrm{Tate}}. \] Its representations describe polarizable Hodge structures modulo Tate twists, with the usual Tannakian completion of the orbit category. Explicitly, for a specific $\mathbb{Q}$-polarizable pure Hodge structure $(V,h,\varphi)$, its associated Hodge group $\mathsf{Hod}(V)$ is the smallest algebraic subgroup $H \subset \mathrm{GL}(V)$ defined over $\mathbb{Q}$ such that the restriction $h|_{\mathbb{U}_1}\colon\mathbb{U}_1\rightarrow \mathrm{GL}(V)_{\mathbb{R}}$ factors through $H_{\mathbb{R}}$. Here, $\mathbb{U}_1 \subset \mathbb{S}$ is the real algebraic torus whose real points form the circle group $\mathbb{U}_1(\mathbb{R})=\left\{z\in \mathbb{C}^{\times}:|z|=1\right\}$. The torus $\mathbb U_1$ is the kernel of the norm map $\mathbb S\to\mathbb G_{m,\mathbb R}$. On $\Q(m)$, the $\mathbb S$-action is the character $(z\bar z)^m$, whose restriction to $\mathbb U_1$ is trivial. Consequently, Tate twisting does not change the restricted $\mathbb U_1$-representation or its rational algebraic envelope. 
 
 Let $\mathsf{Hod}(X)$ be the algebraic image of $\mathsf{Hod}$ in
 $\mathrm{GL}(\rH^*(X,\Q))$. It is a connected reductive algebraic
 group of finite type, and
 \[
 \rH^*(X,\Q)^{\mathsf{Hod}(X)}=H_{\Hdg}(X).
 \]
 Indeed, the circle acts on a class of type $(p,q)$ through
 $z^{q-p}$, so its fixed vectors have $p=q$.
 By reductivity, there is a graded invariant complement
 $V_{\mathrm{nt}}$ with
 \[
 \rH^{\mathrm{ev}}(X,\Q)=H_{\Hdg}(X)\oplus V_{\mathrm{nt}}.
 \]
 This justifies the choice of homogeneous basis made in
 Section~\ref{sec:A-model}: first extend the unit and divisor classes
 $T_0,\ldots,T_\rho$ to a basis of $H_{\Hdg}(X)$, then choose a
 basis of $V_{\mathrm{nt}}$.
 
 \subsubsection*{Hodge group action on the base of the A-model F-bundle}
 
 After base change to $\Bbbk$, the analytification of $\mathsf{Hod}(X)_\Bbbk$ has an analytic group germ at the identity. This germ acts on $B_X$, fixing the divisor coordinates and acting linearly on the additive coordinates, as in \cite[\S\S5.2.1--5.2.2]{KKPY}.
 
 We define $B_X^{\Hod}\subset B_X$ as the fixed locus of this analytic group germ, which corresponds to the connected, smooth Hodge-fixed base introduced in \cite[\S5.4]{KKPY}. In coordinates, this is the locus where all additive coordinates outside $H_{\Hdg}(X)$ vanish and all divisor coordinates are kept. The evaluation pushforwards \[   (\operatorname{ev}_1,\ldots,\operatorname{ev}_r)_*   [\overline{\mathcal M}_{0,r}(X,\beta)]^{\mathrm{vir}} \] are rational algebraic cycle classes on $X^r$. Their cohomology classes are therefore invariant under the diagonal Hodge-group action. The Poincar\'e pairing is also invariant under $\mathsf{Hod}(X)$. Contracting these invariant tensors shows that the quantum product is $\mathsf{Hod}(X)$-equivariant. At a point of $B_X^{\Hod}$, it consequently preserves $H_{\Hdg}(X)\otimes_\Q\Bbbk$. Moreover, $\Eu$ is invariant on this locus and $\Grad$ preserves $H_{\Hdg}(X)$. Hence, the quantum connection restricts to $\cH_{\Hdg}$. Consequently, the restriction
 \[
 \cH_{\Hdg}:=H_{\Hdg}(X)\otimes_\Q\cO_{B_X^{\Hod}\times\mathbb D}
 \]
 forms a purely even F-bundle. At every rational point $b$ of this base, the tangent space identifies with the invariant subspace of $T_bB_X$, yielding $T_bB_X^{\Hod}\cong H_{\Hdg}(X)\otimes_\Q\Bbbk$. Evaluating the quantum product on the identity class $\mathbf1$ verifies maximality.
 
 \begin{lemma}[{\cite[Lem.~5.19]{KKPY}}]\label{lem:reduced-spectrum}
  At a rational point $b\in B_X^{\Hod}(\Bbbk)$, the
  operator $K_b^{\mathrm F}$ has the same distinct eigenvalues on
  $\rH^*(X,\Bbbk)$ as on
  $H_{\Hdg}(X)\otimes_\Q\Bbbk$.
 \end{lemma}

 In many of our arguments, we restrict our computations to small quantum multiplication. Next, we describe this locus within the A-model base. We also introduce a specialized F-bundle that naturally governs the quantum differential equations of complete intersections, following \cite[\S6.1]{KKPY}.
 
 \subsection{The small locus and ambient classes}\label{sec:small}
 
 The \emph{small locus} is the analytic subspace of $B_X$ defined by $y_i=0$ for every additive coordinate; it is naturally identified with $U_X$. Suppose that $X$ is a Fano variety of positive dimension. Write $c_1(TX)=rP$, where $P$ is a primitive integral ample divisor class and $r>0$ is the Fano index. Expressing $P=\sum_m p_m T_m$, we define the \emph{one-parameter small locus along the primitive anticanonical ray} by the relations
 \begin{equation}\label{eq:very-small}
  q_m=q^{p_m},\qquad\text{so that}\qquad q^\beta=\prod_m q_m^{\int_\beta T_m}=q^{\int_\beta P},
  \qquad 0<|q|<1.
 \end{equation}
 By applying the chain rule $q\partial_q = \sum_m p_m q_m\partial_{q_m}$, the full quantum connection \eqref{eq:quantum-connection} restricted to this locus reduces to
 \begin{equation}\label{eq:small}
  \nabla_{\partial_u}=\partial_u+\frac{K^{\mathrm F}}{u^2}+\frac{G}{u},
  \qquad
  \nabla_{\partial_q}=\partial_q+\frac{P\star}{uq},
  \qquad
  K^{\mathrm F}=-c_1(TX)\star=-r(P\star),
  \quad G=\Grad.
 \end{equation}
 On this ray, the structure constants of the small quantum product lie in $\Q[q]$. Indeed, the dimension constraint fixes $c_1(TX)\cdot\beta$ for each homogeneous structure constant. Ampleness bounds the contributing effective curve classes, and only finitely many contribute in each such degree.

 Suppose, in addition, that the Fano variety $X$ is
 the smooth zero locus of a regular section of a vector
 bundle $\cV$ on a smooth projective variety $Y$. We write
 $\iota\colon X\hookrightarrow Y$ for the canonical inclusion. Let
 $\mathbf R:=\operatorname{Im}(\iota^*\colon\rH^*(Y,\Q)\to\rH^*(X,\Q))$ be its
 graded \emph{ambient} subring. Adjunction yields
 $c_1(TX)=\iota^*(c_1(TY)-c_1(\cV))\in\mathbf R$.
 Since $P=c_1(TX)/r$, both $c_1(TX)$ and $P$ belong to $\mathbf R$. If $\mathbf R$ is closed under the small quantum product, multiplication by either class preserves $\mathbf R$. The grading operator also preserves $\mathbf R$, because $\iota^*$ is graded. Hence $\mathbf R\otimes_\Q\cO$ is preserved by the connection \eqref{eq:small} and defines an analytic sub-F-bundle over \eqref{eq:very-small}. This sub-F-bundle need not be maximal. Its operators are the restrictions $K^{\mathrm F}|_{\mathbf R}$ and $G|_{\mathbf R}$.
 
 \section{Formal exponential type}
 \label{sec:formal}
 
 We use the formal local theory of meromorphic connections \cite[Ch.~3]{PutSinger}.
 
 Write $\Bbbk[[u]]$ for the $u$-adic completion of the ring
 $\cO_{\mathbb D,0}$ of analytic germs at $u=0$, and
 $\Bbbk((u)):=\Bbbk[[u]][u^{-1}]$ for its fraction field. For an F-bundle $(\cH,\nabla)/B$ and a $\Bbbk$-rational point $b\in B(\Bbbk)$, put
 \[
 \widehat\cH_b:=(\cH|_{\{b\}\times\mathbb D})_0
 \otimes_{\cO_{\mathbb D,0}}\Bbbk[[u]],
 \qquad {\widehat \cH}_b^{\mathrm{mer}}:=\widehat\cH_b[u^{-1}].
 \] The spaces $\widehat\cH_b$ and ${\widehat \cH}_b^{\mathrm{mer}}$ are fibers at $b$ of infinite-dimensional bundles $\widehat\cH$, $\widehat\cH^{\mathrm{mer}}$ on $B$, and are finitely-generated free modules over $\Bbbk[[u]]$ and $\Bbbk((u))$, respectively. Operators of multiplication by $u$
  and of covariant derivative $\nabla_{u^2 \partial_u}$ are $\mathcal O_B$-linear operators on these two infinite-dimensional bundles on $B$. Moreover, $\widehat\cH^{\mathrm{mer}}$ carries a flat connection along $B$.
 
 For an endomorphism $A$, $\spec(A)$
 denotes its set of distinct eigenvalues and
 $\mult_{\mathrm{alg}}(\lambda;A)$ their algebraic multiplicities.
 Double braces $\ms{\cdots}$ denote multisets and $\cl\alpha:=\alpha+\Z\in\Bbbk/\Z$.
 
 \subsection{Logarithmic lattices and exponents}
 \label{def:reg-sing}
 A \emph{meromorphic connection} over $\Bbbk((u))$ is a finite-dimensional $\Bbbk((u))$-vector space $\cM$ equipped with a $\Bbbk$-linear operator $\nabla_{\partial_u}\colon \cM \to \cM$ satisfying the Leibniz rule $\nabla_{\partial_u}(fv) = \partial_u(f)v + f\nabla_{\partial_u}(v)$ for all $f \in \Bbbk((u))$ and $v \in \cM$. A \emph{lattice} is a finitely generated free $\Bbbk[[u]]$-submodule $\cL\subset\cM$ that spans $\cM$ over $\Bbbk((u))$, meaning $\cL[u^{-1}] = \cM$. We say a lattice $\cL$ is \emph{logarithmic} if $u\nabla_{\partial_u}\cL\subseteq\cL$. The connection is \emph{regular singular} if it admits a logarithmic lattice $\cL$. In this case, the \emph{residue} of the connection with respect to $\cL$ is the well-defined $\Bbbk$-linear endomorphism of the finite-dimensional vector space $\cL/u\cL$ induced by the operator $u\nabla_{\partial_u}$. Because $\Bbbk[[u]]$ is a discrete valuation ring, any two lattices in $\cM$ are commensurable.
 \begin{lemma}[{\cite[\S3.1.1, Lem.~3.10, and Rem.~3.20(1)]{PutSinger}}]\label{lem:lattice-independence}
  Let $(\cM,\nabla)$ be a regular-singular meromorphic connection over
  $\Bbbk((u))$.
  \begin{enumerate}
   \item The multiset in $\Bbbk/\Z$ of residue eigenvalues modulo $\Z$, counted
   with algebraic multiplicity, does not depend on the choice of
   logarithmic lattice.
   \item Every differential submodule $\cN\subseteq\cM$ is regular singular.
   Its multiset of residue eigenvalues modulo $\Z$ is a submultiset of
   that of $\cM$.
  \end{enumerate}
 \end{lemma}
 
 \begin{definition}\label{def:exponents}
  For $\lambda\in\Bbbk$, define \[   \cE^{\lambda/u}:=\Bbbk((u))e_\lambda,   \qquad   \nabla_{\partial_u}e_\lambda   =-\lambda u^{-2}e_\lambda. \] Thus \[   \nabla_{\partial_u}(f e_\lambda)   =(f'-\lambda u^{-2}f)e_\lambda. \] After adjoining an invertible differential symbol $s_\lambda$ satisfying $s_\lambda'=\lambda u^{-2}s_\lambda$, the section $s_\lambda e_\lambda$ is horizontal. We write $s_\lambda=e^{-\lambda/u}$. 
  
  A meromorphic connection $\cM$ is said to be of \emph{unramified exponential type} if it decomposes as
  \begin{equation}\label{eq:elementary}
   \cM\cong\bigoplus_{\lambda\in\Lambda}
   \cE^{\lambda/u}\otimes\cR_\lambda
  \end{equation}
  for a finite indexing set $\Lambda\subset\Bbbk$ and a collection of regular-singular connections $\cR_\lambda$. 
  
  For a regular-singular connection $\cR$, we define its multiset of exponents:
  \[
  \Exp(\cR):=\bigl\{\bigl\{ -\vartheta+\Z \bigm| \vartheta\text{ is a residue eigenvalue of } \cR \bigr\}\bigr\},
  \]
  where the elements are counted with their algebraic multiplicities. We extend this definition to the elementary factors setting $\Exp(\cE^{\lambda/u}\otimes\cR_\lambda):=\Exp(\cR_\lambda)$.
  
  A regular-singular connection has \emph{rational residue eigenvalues} if the residue eigenvalues of one, equivalently every, logarithmic lattice lie in $\Q$. We abbreviate this condition to \emph{rational residues}. We say that a connection of unramified exponential type, as in \eqref{eq:elementary}, has rational residues if this condition holds for every regular-singular component $\cR_\lambda$. Because changing the underlying logarithmic lattice only shifts the residue eigenvalues by integers, Lemma~\ref{lem:lattice-independence} ensures that both the rationality condition and the multiset of exponents $\Exp$ are well-defined and lattice-independent.
 \end{definition}
 
 In a decomposition \eqref{eq:elementary} with nonzero summands,
 the exponential parameters and their differential summands are intrinsic; the regular-singular factors are unique up to isomorphism
 \cite[Rem.~3.55(1)]{PutSinger}.
 
 \begin{lemma}[{\cite[Prop.~3.19 and Rem.~3.20(1)]{PutSinger}}]
  \label{lem:splitting}
  Let $F$ be an algebraically closed field of characteristic zero, and let $\cL$ be an $F[[u]]$-lattice in a meromorphic connection $\cM$ over $F((u))$ such that $u^2\nabla_{\partial_u}\cL\subseteq\cL$. Define the leading endomorphism
  \[
  A_0:=(u^2\nabla_{\partial_u})\bmod u\cL\in\End_F(\cL/u\cL),
  \]
  and consider the generalized eigenspace decomposition $\cL/u\cL=\bigoplus_{\lambda}V_\lambda$, where $V_\lambda$ corresponds to the eigenvalue $\lambda \in \spec(A_0)$. Then:
  \begin{enumerate}
   \item There exists a unique decomposition $\cL=\bigoplus_\lambda\cL_\lambda$ into $u^2\nabla_{\partial_u}$-stable $F[[u]]$-submodules such that the reduction modulo $u$ yields $\cL_\lambda/u\cL_\lambda=V_\lambda$.
   \item This lattice decomposition induces a direct sum decomposition of the meromorphic connection $\cM=\bigoplus_\lambda\cM_\lambda$, where each $\cM_\lambda:=\cL_\lambda[u^{-1}]$ is a differential submodule of rank $\rk\cM_\lambda=\dim V_\lambda$.
   \item For each eigenvalue $\lambda$, the connection $\cE^{\lambda/u}\otimes\cM_\lambda$ has a nilpotent leading endomorphism given by $(A_0-\lambda)|_{V_\lambda}$. Consequently, the full connection factors as
   \[
   \cM\cong\bigoplus_\lambda\cE^{-\lambda/u}\otimes\bigl(\cE^{\lambda/u}\otimes\cM_\lambda\bigr).
   \]
  \end{enumerate}
 \end{lemma}

 \begin{corollary}\label{cor:leading-exponential} In the setting of Lemma~\ref{lem:splitting}, suppose that $\cM$ is of unramified exponential type: \[   \cM=\bigoplus_{\lambda\in\Lambda}\cN_\lambda,   \qquad   \cN_\lambda\cong\cE^{\lambda/u}\otimes\cR_\lambda, \] where every $\cR_\lambda$ is regular singular. Then \[   \Lambda=-\spec(A_0),   \qquad   \cM_\lambda=\cN_{-\lambda},   \qquad   \rk\cR_{-\lambda}=\dim V_\lambda. \] \end{corollary} 
 \begin{proof} Choose a logarithmic lattice in each $\cR_\lambda$. The induced lattice in $\cN_\lambda$ has leading endomorphism $-\lambda I$. We first verify that a differential morphism $\phi: \cM_{\lambda_1} \to \cN_{\lambda_2}$ between two connections with quadratic lattices and disjoint leading spectra is zero, even if the morphism does not preserve the lattices. In chosen lattice bases, write the morphism as a Laurent matrix \[   \phi(u)=u^j\phi_j+u^{j+1}\phi_{j+1}+\cdots,   \qquad \phi_j\neq0. \] If the two connection matrices are $A(u)/u^2$ and $B(u)/u^2$, horizontality gives \[   u^2\phi'(u)+B \phi(u)-\phi(u) A=0. \] The coefficient of $u^j$ is \[   B_0\phi_j-\phi_jA_0=0. \] Disjoint leading spectra make this Sylvester operator invertible, a contradiction. Apply this observation to the components of the identity map between $\cM=\bigoplus_{\lambda_1} \cM_{\lambda_1}$ and $\cM=\bigoplus_{\lambda_2}\cN_{\lambda_2}$. Every component vanishes unless $\lambda_1=-\lambda_2$. Applying the same argument in the reverse direction identifies $\cM_{\lambda_1}$ with the embedded summand $\cN_{-\lambda_1}$. The rank formula follows from Lemma~\ref{lem:splitting}. \end{proof}

 \begin{lemma}\label{lem:descent}
  Let $F\subseteq F'$ be algebraically closed fields of characteristic zero and
  let $A_1, A_2$ be matrices over $F$. Then $\partial_u+A_2u^{-2}+A_1u^{-1}$ is of
  unramified exponential type with rational residues over $F((u))$ if and only
  if its extension to $F'((u))$ is.
 \end{lemma}
 \begin{proof}
  Apply Lemma~\ref{lem:splitting} over $F$ to the natural lattice.
  Its uniqueness identifies the scalar extensions of its spectral
  summands with those supplied by the same lemma over $F'$.
  For an eigenvalue $\lambda$ of $A_2$, set
  $\cN_\lambda:=\cE^{\lambda/u}\otimes\cM_\lambda$.
  By Corollary~\ref{cor:leading-exponential}, unramified exponential
  type is equivalent to regular singularity of every $\cN_\lambda$.
  It remains to descend that property and rationality of residues.
  
  Fix $\cN:=\cN_\lambda$, an $F[[u]]$-lattice $\mathcal L_0$
  with basis $e_1,\ldots,e_r$, and set $D:=u\nabla_{\partial_u}$.
  If the extended connection is regular singular, choose a
  logarithmic $F'[[u]]$-lattice $\mathcal L'$ in it. Write
  $\mathcal L'_0:=\mathcal L_0\otimes_{F[[u]]}F'[[u]]$.
  Commensurability gives integers $a,b\geq0$ such that
  \[
  u^b\mathcal L'_0\subseteq\mathcal L'
  \subseteq u^{-a}\mathcal L'_0.
  \]
  The shifted lattice $u^{-b}\mathcal L'$ is also $D$-stable,
  contains $\mathcal L'_0$, and lies in
  $u^{-(a+b)}\mathcal L'_0$. Consequently, for every $j\geq0$
  and $1\leq\ell\leq r$,
  \[
  D^je_\ell\in u^{-(a+b)}\mathcal L'_0
  \cap\bigl(\mathcal L_0\otimes_{F[[u]]}F((u))\bigr)
  =u^{-(a+b)}\mathcal L_0.
  \]
  The $F[[u]]$-submodule
  \[
  \mathcal L_\infty:=
  \sum_{j\geq0,\,1\leq\ell\leq r}F[[u]]D^je_\ell
  \]
  is therefore contained in a finite free module. Since $F[[u]]$
  is Noetherian, it is finitely generated; since it contains $\mathcal L_0$ and is torsion-free, it is a full free lattice.
  The Leibniz rule makes it $D$-stable. Thus $\cN$ is regular
  singular over $F((u))$. Conversely, a logarithmic lattice
  extends to one over $F'[[u]]$.
  
  The residue of $\mathcal L_\infty$ has entries in $F$.
  After extension, its eigenvalues represent the same classes
  modulo $\Z$ as those of $\mathcal L'$ by
  Lemma~\ref{lem:lattice-independence}. If the latter are rational,
  then each eigenvalue of the former belongs to $\Q$, since it
  differs from a rational number by an integer. The same lemma
  identifies the exponent multiset before and after extension.
 \end{proof}
 
 \subsection{Fano specializations}
 
 Our main goal in this subsection is to prove the following.
 
 \begin{theorem}\label{thm:fano}
  Let $X$ be a smooth projective Fano variety and let $b$ be a rational point
  of the one-parameter small locus \eqref{eq:very-small}. Then the formal
  disk connection $\cH_b^{\mathrm{mer}}$ is of unramified exponential type
  with rational residues, and
  \begin{equation}\label{eq:fano-decomposition}
   {\widehat \cH}_b^{\mathrm{mer}}
   =\bigoplus_{\lambda\in\spec(K_b^{\mathrm F})}\cE^{-\lambda/u}\otimes\cR_{\lambda,b},
   \qquad
   \rk\cR_{\lambda,b}=\mult_{\mathrm{alg}}(\lambda;K_b^{\mathrm F}).
  \end{equation}
  The same assertion holds for a graded ambient subspace
  closed under the small quantum product, with spectrum
  and algebraic multiplicities computed for
  $K_b^{\mathrm F}|_{\mathbf R}$. On the Hodge-fixed part, these formal summands are the completed
  fibers of the local spectral decomposition given by Theorem \ref{thm:hyzz}.
 \end{theorem}
 
 Theorem \ref{thm:fano} is proven via the following result of Chen (Theorem \ref{thm:chen}). Earlier work of Pomerleano--Seidel
 \cite{PS} established the same under an additional geometric hypothesis.
 
 \begin{theorem}[{\cite[Thm.~1.3]{Chen}}]\label{thm:chen}
  Let $(M,\omega)$ be a closed monotone symplectic manifold of real dimension
  $2\nu$, and let $\star_1$ be its small quantum product with the monotone Novikov parameter specialized to $1$. The connection
  \[
  \partial_u-\frac{c_1(TM)\star_1}{u^2}+\frac{\mu}{u},
  \qquad \mu|_{\rH^a(M)}=\frac{a-\nu}{2}\,\id,
  \]
  on $\rH^*(M;\C)((u))$ is of unramified exponential type and has
  quasi-unipotent regularized monodromies.
 \end{theorem}
 
 \begin{proof}[Proof of Theorem \ref{thm:fano}]
  Let $\{\sigma_i\}$ be a homogeneous basis, with
  $d_i:=\deg\sigma_i$.
  To obtain the coefficient of $\sigma_i$ in
  $c_1(TX)\star\sigma_j$, pair with a class of degree
  $2n-d_i$. The three-point dimension constraint is
  \[
  2+d_j+(2n-d_i)=2n+2c_1(TX)\cdot\beta.
  \]
  Writing $a:=P\cdot\beta$ gives
  \begin{equation}\label{eq:homogeneity}
   d_i+2ra=d_j+2.
  \end{equation}
  Thus each nonzero matrix entry of $K_b^{\mathrm F}=-c_1(TX)\star_b$ is
  $\kappa_{ij}q^{a_{ij}}$, with
  $\kappa_{ij}\in\Q$ and
  $a_{ij}=(d_j+2-d_i)/(2r)\in\Z_{\geq0}$.
  
  Choose $\alpha\in\Bbbk^\times$ satisfying
  $\alpha^{2r}=q(b)$ and put
  $D:=\operatorname{diag}(\alpha^{d_i})$.
  Then
  \[
  (DK_b^{\mathrm F}D^{-1})_{ij}
  =\kappa_{ij}\alpha^{d_i+2ra_{ij}-d_j}
  =\alpha^2\kappa_{ij},
  \]
  so $DK_b^{\mathrm F}D^{-1}=\alpha^2K_1^{\mathrm F}$, where
  $K_1^{\mathrm F}$ denotes the algebraic specialization at $q=1$.
  Since $D$ and $G$ are diagonal in this basis,
  $DGD^{-1}=G$.
  The constant gauge transforms the connection into
  \[
  \partial_u+\frac{\alpha^2K_1^{\mathrm F}}{u^2}+\frac Gu.
  \]
  With $v=u/\alpha^2$, one has
  $\partial_v=\alpha^2\partial_u$, and hence
  \[
  \nabla_{\partial_v}
  =\alpha^2\nabla_{\partial_u}
  =\partial_v+\frac{K_1^{\mathrm F}}{v^2}+\frac Gv.
  \]
  
  More generally, under $v=cu$,
  \[
  \left(-\frac{\lambda}{u^2}+\frac{\vartheta}{u}\right)du
  =
  \left(-\frac{c\lambda}{v^2}+\frac{\vartheta}{v}\right)dv.
  \]
  Thus exponential parameters are multiplied by $c$,
  whereas residue eigenvalues are unchanged.
  It suffices to treat the algebraic specialization $q=1$ of
  the polynomial connection matrix. This is not a point of the
  non-Archimedean small locus $0<|q|<1$.
  
  Embed $\overline\Q\hookrightarrow\C$ and choose a K\"ahler form representing
  $c_1(TX)$. The resulting symplectic manifold is monotone; its small quantum
  product at Novikov parameter $1$ coincides with the algebraic one
  \cite{Siebert}, and $G$ is the grading operator in Theorem \ref{thm:chen}. By Theorem \ref{thm:chen}, the
  connection is of unramified exponential type with quasi-unipotent regularized
  monodromies. A residue eigenvalue $\vartheta$ gives the monodromy eigenvalue
  $e^{-2\pi i\vartheta}$, which is a root of unity, so $\vartheta\in\Q$. Applying
  Lemma \ref{lem:descent} to $\overline\Q\subset\C$ and then to
  $\overline\Q\subset\Bbbk$, and undoing the rescaling and the gauge, gives the
  claim at $b$.
  
  On the natural lattice $\rH^*(X,\Bbbk)[[u]]$,
  the leading endomorphism is $K_b^{\mathrm F}$.
  Corollary~\ref{cor:leading-exponential} therefore
  identifies the exponential parameters with the
  negatives of the eigenvalues of $K_b^{\mathrm F}$ and gives the rank formula
  in \eqref{eq:fano-decomposition}. For a graded ambient subspace preserved by $K_b^{\mathrm F}$, apply
  Lemma~\ref{lem:splitting} to its natural lattice as well.
  The Laurent-matrix argument in the proof of Corollary~\ref{cor:leading-exponential} shows that the inclusion
  has no components between distinct leading eigenvalues.
  After the corresponding exponential twists, its summands are
  differential submodules of the $\cR_{\lambda,b}$.
  Lemma~\ref{lem:lattice-independence} gives regular singularity
  and rational residue eigenvalues for each of them.
  
  Finally, $\cH_{\Hdg}$ is maximal, so Theorem \ref{thm:hyzz} yields a local product of
  bases and an external sum lifting the generalized eigenspaces of
  $K_b^{\mathrm F}$. Completion at $u=0$ and uniqueness of the formal
  splitting identify the fibers of the factors with the Hodge-fixed summands of
  \eqref{eq:fano-decomposition}.
 \end{proof}
 
 \subsection{Nef specialization}
 
 We next study smooth projective varieties with nef canonical class. Proposition~\ref{prop:nef} refines the regular-singularity argument of \cite[Claim~6.15]{KKPY} by identifying the exponents.
 
 \begin{proposition}\label{prop:nef}
  Let $W$ be a connected smooth projective variety of dimension $d$ with canonical class $K_W$
  nef, and let $b\in B_W^{\Hod}(\Bbbk)$. The disk connection at $b$ has the
  single exponential parameter $y_0(b)$. Moreover, after twisting by $\cE^{-y_0(b)/u}$, the connection is
  regular singular with rational residues, and on the Hodge-fixed part
  \begin{equation}\label{eq:nef-exponents}
   \rk\cR_{\Hdg}=\dim_\Q H_{\Hdg}(W),
   \qquad
   \Exp(\cR_{\Hdg})=\ms{\cl{d/2}^{\,[\dim H_{\Hdg}(W)]}},
  \end{equation}
  where the superscript $[m]$ denotes $m$ repetitions.
 \end{proposition}
 \begin{proof}
  Let $b'$ be obtained from $b$ by setting $y_0=0$.
  The string equation gives $\star_b=\star_{b'}$; since
  $\partial_{y_0}\Eu=\mathbf1$, it also gives
  $K_b^{\mathrm F}=K_{b'}^{\mathrm F}-y_0(b)\id$. Thus, after tensoring by
  $\cE^{-y_0(b)/u}$, we may assume $y_0(b)=0$.
  Consider a nonzero contribution to the product of
  classes of degrees $2p$ and $a$, with curve class
  $\beta$ and additional insertions of degrees
  $2p_1,\ldots,2p_k$.
  If the output has degree $h$, its Poincar\'e-dual test class has degree $2d-h$. The dimension axiom
  therefore gives
  \[
  2p+a+(2d-h)+2\sum_{\nu=1}^k p_\nu
  =2\bigl(d+c_1(TW)\cdot\beta+k\bigr).
  \]
  Solving for $h$ yields
  \begin{equation}\label{eq:degree-nef}
   h=a+2p+2\sum_{\nu=1}^k(p_\nu-1)
   -2c_1(TW)\cdot\beta.
  \end{equation}
  Because $b\in B_W^{\Hod}$, all remaining additive insertions
  are rational Hodge classes. The unit coordinate has been set
  to zero, and all degree-two Hodge directions belong to the
  N\'eron--Severi space and have been absorbed into the
  Novikov variables. Thus $p_\nu\geq2$; every term of $\Eu$ has
  $p\geq1$ since $y_0=0$; and $c_1(TW)\cdot\beta\leq0$. Hence $K^{\mathrm F}=-\Eu\star$
  strictly raises degree, and we write $K^{\mathrm F}=\sum_{j\geq1}K_j^{\mathrm F}$ with $K_j^{\mathrm F}$ raising
  degree by $2j$.
  
  On $\rH^{2p+\epsilon}(W,\Bbbk)$, with
  $\epsilon\in\{0,1\}$, define $\Gamma:=p\,\id$.
  On the parity-$\epsilon$ part,
  \[
  G=\Gamma+\frac{\epsilon-d}{2}\id.
  \]
  Since $K_j^{\mathrm F}$ raises degree by $2j$,
  \[
  u^\Gamma K_j^{\mathrm F} u^{-\Gamma}=u^jK_j^{\mathrm F}.
  \]
  Moreover, applying both sides to a section gives
  \[
  u^\Gamma\partial_u u^{-\Gamma}
  =\partial_u-\frac{\Gamma}{u}.
  \]
  Consequently,
  \begin{equation}\label{eq:nef-gauge}
   u^\Gamma\nabla_{\partial_u}u^{-\Gamma}
   =
   \partial_u+\sum_{j\geq1}u^{j-2}K_j^{\mathrm F}
   +\frac{\epsilon-d}{2u}\id.
  \end{equation}
  Only the term $j=1$ contributes to the residue.
  Thus, the transformed connection is logarithmic,
  with residue:
  \[
  K_1^{\mathrm F}+\frac{\epsilon-d}{2}\id.
  \]
  The operator $K_1^{\mathrm F}$ raises degree by two, so
  $(K_1^{\mathrm F})^{d+1}=0$. Every residue eigenvalue on the
  parity-$\epsilon$ part is therefore
  $(\epsilon-d)/2$, and our convention for exponents
  gives the class $\cl{(d-\epsilon)/2}$.
  
  Equivalently,
  $u^{-\Gamma}\rH^*(W,\Bbbk)[[u]]$
  is a logarithmic lattice for the connection after
  the scalar exponential twist.
  Returning to the original point $b$, put
  \[
  \cR_{\Hdg}:=
  \cE^{-y_0(b)/u}\otimes
  {\widehat \cH}_{\Hdg,b}^{\mathrm{mer}}.
  \]
  The gauge preserves the Hodge subspace, which is
  purely even. Hence $\cR_{\Hdg}$ has rank
  $\dim_\Q H_{\Hdg}(W)$ and all its exponents are
  $\cl{d/2}$, as asserted.
  The full connection is
  $\cE^{y_0(b)/u}$ tensored with a nonzero regular-singular connection, so its only exponential parameter is $y_0(b)$.
 \end{proof}
 
 \subsection{Local persistence of formal unramified exponential type}
 \label{sec:nef-deformation}
 
 Let $(\cH,\nabla)/B_X$ be the A-model F-bundle of a connected
 smooth complex projective variety $X$. The next lemma concerns
 cyclicity of the entire Hodge-invariant tangent algebra; in
 Proposition~\ref{prop:cyclic-transport} we apply the same leaf
 argument to a single maximal spectral factor.
 
 \begin{lemma}\label{lem:cyclic}
  Put $N:=\dim_\Q H_{\Hdg}(X)$, and let
  $b\in B_X^{\Hod}(\Bbbk)$. Assume that
  \[
  \mathbf1,\Eu_b,\Eu_b^{\star2},\ldots,\Eu_b^{\star(N-1)}
  \]
  form a basis of $T_bB_X^{\Hod}$.
  Equivalently, the minimal polynomial of
  $K_b^{\mathrm F}|_{H_{\Hdg}(X)\otimes_\Q\Bbbk}$ has degree $N$.
  Then there is an admissible neighborhood $U\subset B_X^{\Hod}$
  of $b$ such that:
  \begin{enumerate}
   \item On $\cH_{\Hdg}|_{U\times\{0\}}$, the number of distinct
   eigenvalues of $K^{\mathrm F}$, their algebraic multiplicities, and their Jordan block sizes are constant.
   The number of distinct eigenvalues on full cohomology is also
   constant on $U$.
   \item If ${\widehat \cH}_b^{\mathrm{mer}}$ is of unramified exponential
   type, then so is ${\widehat \cH}_{b'}^{\mathrm{mer}}$ for every
   $b'\in U(\Bbbk)$. After matching the eigenvalues by the local
   spectral decomposition, the corresponding regular-singular
   factors are isomorphic over $\Bbbk((u))$.
  \end{enumerate}
 \end{lemma}
 \begin{proof}
  In a unital algebra, a polynomial in multiplication by an
  element vanishes when its value on the unit vanishes.
  Since $K_b^{\mathrm F}$ is multiplication by $-\Eu_b$, this proves
  the stated equivalence. The cyclic determinant remains nonzero in a neighborhood of $b$, so these powers give a local frame.
  By \eqref{eq:sign}, the powers of $\Eu^{\mathrm F}=-\Eu$ give
  the same frame. The Euler leaf of the maximal Hodge-invariant
  F-bundle in Theorem~\ref{cl:formal-persistence} is consequently
  a neighborhood of $b$ in $B_X^{\Hod}$.
  Theorem~\ref{cl:formal-persistence} gives local constancy of the spectral multiplicities.
  Cyclicity gives one Jordan block for each eigenvalue, whose size
  is its multiplicity. Lemma~\ref{lem:reduced-spectrum} provides
  the assertion about the full spectrum.
  
  For the formal-type assertion on full cohomology, use
  Remark~\ref{rem:compatible-persistence}. Indeed, for
  $X_p=p_\star(-\Eu)$, associativity gives on the full bundle
  \[
  \nabla_{uX_p}|_{u=0}
  =p_\star(-\Eu)\star(-)=p(K^{\mathrm F}).
  \]
  Thus, the compatibility required there holds along the Hodge leaf,
  and the same transport preserves all regular-singular factors.
 \end{proof}
 
 \section{Blowups, projective bundles, and an irrationality criterion}
 \label{sec:criterion}
 
 The construction of Hodge atoms in \cite[\S5]{KKPY} motivates the study
 of formal connection invariants under the comparisons associated with
 blowups and projective bundles. In this section, we consider generic
 spectral factors of the Hodge-invariant A-model F-bundle whose leading
 endomorphism is cyclic. We prove that unramified exponential type with
 rational residues propagates along the cyclic locus of a connected
 generic spectral component, and that the exponent multiset of the
 regularized factor is constant there. We then transport these
 invariants through the comparison domains and apply weak factorization
 to obtain an irrationality criterion in dimensions two, three, and four.
 Gong and Xu \cite[\S3]{gong2026irrationalitystabilizationcomplexsmooth}
 study related questions using Gu\'er\'e's methods
 \cite{guere2026irrationalitycubicfourfolds}.

 \ 
 
 Let $W$ be a connected smooth complex projective variety. Denote by
 $\mathcal F_W$ the maximal Hodge-invariant A-model F-bundle over the
 purely even base $B_W^{\Hod}$. Its fiber at $u=0$ is
 \[
 \mathcal A_W
 :=H_{\Hdg}(W)\otimes_{\Q}\cO_{B_W^{\Hod}},
 \]
 equipped with the quantum product. Let $U_W^{\mathrm{gen}}\subset B_W^{\Hod}$ be the open locus where
 $K^{\mathrm F}$ has the maximal number of distinct eigenvalues, and
 let
 \[
 \pi_W\colon\widetilde U_W^{\mathrm{gen}}
 \longrightarrow U_W^{\mathrm{gen}}
 \]
 be the reduction, over $U_W^{\mathrm{gen}}$, of the finite analytic space defined by \[   \det(z\id-K^{\mathrm F})=0 \] in $U_W^{\mathrm{gen}}\times\mathbb A^{1,\mathrm{an}}_\Bbbk$. This cover is finite and \'etale; its
 geometric fibers consist of the distinct eigenvalues, without their
 multiplicities. By Lemma~\ref{lem:reduced-spectrum}, the full cohomology
 bundle has the same reduced spectral cover.
 
 Fix a connected component $S$ of
 $\widetilde U_W^{\mathrm{gen}}$, and let $\lambda\in\cO(S)$ be the
 tautological eigenvalue. All bundles and endomorphisms below are
 implicitly pulled back to $S$. The generalized $\lambda$-eigenspace
 bundle has a constant rank $m$. For $s\in S(\Bbbk)$, write
 $\mathcal M_s$ for the corresponding formal meromorphic summand of
 $\mathcal F_W$ at $\pi_W(s)$, and call it the \emph{generic Hodge factor}
 represented by $s$. Its rank is $m$, independently of the degree of
 $S\to U_W^{\mathrm{gen}}$.
 
 Whenever $\mathcal M_s$ is of unramified exponential type, its only
 exponential parameter is $-\lambda(s)$. We write
 \[
 \mathcal R_s
 :=\mathcal M_s\otimes\mathcal E^{\lambda(s)/u},
 \qquad
 F_s:=\operatorname{Exp}(\mathcal R_s).
 \]
 Thus, $\mathcal R_s$ is regular singular, and $F_s$ is the multiset of
 negative residue eigenvalues modulo $\Z$, counted with multiplicity.
 
 The space $S$ is smooth, connected, and $\Bbbk$-analytic.
 We shall use the identity principle in the following form: a
 meromorphic function on $S$ that vanishes on a nonempty analytic open
 subset is identically zero. In particular, the zero locus of a
 nonzero analytic section of a line bundle has empty interior.
 Every nonempty analytic open subset of $S$ contains a
 $\Bbbk$-rational point.
 
 Let $P_0$ be the spectral projector onto the generalized
 $\lambda$-eigenspace, and put
 \[
 \mathcal A:=\operatorname{im}P_0,\qquad
 N:=(K^{\mathrm F}-\lambda\,\id)|_{\mathcal A},\qquad
 v_0:=P_0(\mathbf1).
 \]
 The bundle $\mathcal A$ is a summand algebra with unit $v_0$.
 The endomorphism $N$ is multiplication by
 $n:=Nv_0\in\mathcal A$, and $N^m=0$.
 
 Define the \emph{cyclic locus} by
 \begin{equation}\label{eq:cyclic-locus}
  S^{\mathrm{cyc}}
  :=\left\{s\in S:
  v_0,Nv_0,\ldots,N^{m-1}v_0
  \text{ are linearly independent over }\mathcal H(s)\right\},
 \end{equation}
 where $\mathcal H(s)$ denotes the completed residue field of $s$; notice that this is the nonvanishing locus of the analytic section
 \[
 v_0\wedge Nv_0\wedge\cdots\wedge N^{m-1}v_0
 \in\Gamma(S,\det\mathcal A).
 \]
 Consequently, if $S^{\mathrm{cyc}}$ is nonempty, it is open and dense
 in $S$. 
 
 At a geometric point, the displayed vectors are
 $1,n,\ldots,n^{m-1}$ in the summand algebra. A polynomial relation
 $q(n)=0$ implies $q(N)=0$, since $q(N)$ is multiplication by $q(n)$;
 conversely, $q(N)=0$ implies $q(n)=q(N)v_0=0$. Hence, these vectors
 are independent when $N$ has a single Jordan block.
 
 Write $\delta:=u\partial_u$. We shall use the following scalar
 regular-singularity criterion.
 
 \begin{lemma}[{\cite[Prop.~3.16 and Rem.~3.20(1)]{PutSinger}}]\label{lem:scalar-fuchs}
  Let $\cM$ be a differential module of rank $m$ over $\Bbbk((u))$,
  let $D=\nabla_\delta$, and let $v$ be a cyclic vector, meaning that
  $v,Dv,\ldots,D^{m-1}v$ form a basis over $\Bbbk((u))$, where $m$ is the rank. Write
  \[
  D^mv+\sum_{i=0}^{m-1}a_iD^iv=0,
  \qquad a_i\in\Bbbk((u)).
  \]
  Then $\cM$ is regular singular if and only if
  $a_i\in\Bbbk[[u]]$ for every $i$. In this case,
  \[
  \cL_v:=\sum_{i=0}^{m-1}\Bbbk[[u]]D^iv
  \]
  is a logarithmic lattice, and the characteristic polynomial of its
  residue is
  \[
  T^m+\sum_{i=0}^{m-1}a_i(0)T^i.
  \]
 \end{lemma}
 
 \begin{proposition}\label{prop:cyclic-transport}
  Suppose that $b\in S^{\mathrm{cyc}}(\Bbbk)$ and that $\mathcal M_b$
  is of unramified exponential type with rational residues.
  Then, for every $s\in S^{\mathrm{cyc}}(\Bbbk)$, the regularization
  $\mathcal R_s$ is regular singular and
  \[
  F_s=F_b\subset\Q/\Z.
  \]
  In particular, every nonempty analytic open subset $O\subseteq S$
  contains a point $s\in S^{\mathrm{cyc}}(\Bbbk)$ with these properties.
  Equivalently, if an analytic open subset
  $O_0\subseteq U_W^{\mathrm{gen}}$ has nonempty inverse image in $S$,
  then it contains the image of such a point.
 \end{proposition}
 \begin{proof}
  First, lift the decomposition
  $\mathcal A_W=\operatorname{im}P_0\oplus\ker P_0$ to the
  formal lattice over $S$. Locally, the successive off-diagonal
  terms of the disk connection are removed by gauges $I+u^jT_j$.
  At each order, the unknown off-diagonal matrix $T_j$ is determined
  by the commutator with $K^{\mathrm F}$, which is invertible on
  homomorphisms between the two blocks because their spectra are
  disjoint. Its inverse has analytic coefficients on $S$.
  This gives a decomposition over $\cO_S[[u]]$, unique with the
  prescribed reduction, so the local decompositions glue; this is
  the relative form of Lemma~\ref{lem:splitting}, obtained by the
  argument of \cite[Prop.~3.26]{HYZZ}.
  
  Let $P(u)$ be the projection onto the selected formal summand along
  its complement. Then $P(u)^2=P(u)$ and $P(0)=P_0$.
  After inverting $u$, both summands are preserved by
  $\nabla_\delta$. For $h=h_\lambda+h_{\mathrm{rest}}$, this gives
  \[
  \nabla_\delta(P h)=\nabla_\delta h_\lambda
  =P(\nabla_\delta h),
  \qquad [\nabla_\delta,P]=0.
  \]
  
  Let
  \[
  \widehat\cH:=H_{\Hdg}(W)\otimes_\Q\cO_S[[u]],
  \qquad
  \cL:=P(u)\widehat\cH,
  \qquad
  v:=P(u)\mathbf1.
  \]
  On $\cL[u^{-1}]$ consider
  \[
  D:=\nabla_\delta-\lambda u^{-1}\id.
  \]
  This is the operator of the regularized connection in the
  $u$-direction. Since
  $\nabla_\delta=\delta+K^{\mathrm F}/u+G$, the endomorphism induced by $uD$
  on $\cL/u\cL=\mathcal A$ is $N$.
  
  Put $w_j:=u^jD^jv$. The Leibniz rule gives
  $w_{j+1}=(uD-ju)w_j$. Starting from $w_0=v$, induction yields
  \begin{equation}\label{eq:scaled-cyclic-basis}
   w_j\in\cL,
   \qquad w_j\bmod u=N^jv_0.
  \end{equation}
  On $S^{\mathrm{cyc}}$, the reductions of
  $w_0,\ldots,w_{m-1}$ form a basis. Therefore, these sections form a
  basis of $\cL$ over $\cO_{S^{\mathrm{cyc}}}[[u]]$, and
  $v,Dv,\ldots,D^{m-1}v$ form a basis after inverting $u$.
  
  Write
  \[
  w_m=\sum_{i=0}^{m-1}h_i(s,u)w_i,
  \qquad h_i(s,u)\in\cO(S^{\mathrm{cyc}})[[u]].
  \]
  Dividing by $u^m$ gives
  \begin{equation}\label{eq:cyclic-relation}
   D^mv+\sum_{i=0}^{m-1}a_i(s,u)D^iv=0,
   \qquad a_i=-u^{i-m}h_i.
  \end{equation}
  In particular,
  \begin{equation}\label{eq:finite-fuchs-bounds}
   a_i(s,u)\in
   u^{i-m}\cO(S^{\mathrm{cyc}})[[u]].
  \end{equation}
  
  Every coefficient of every $a_i$ extends meromorphically to $S$.
  Indeed, in a local frame of $\cL$, Cramer's rule expresses each
  $h_i$ as a ratio with denominator
  $\Delta(u):=\det(w_0,\ldots,w_{m-1})$.
  Its constant term represents the cyclic-determinant section,
  which is nonzero at $b$ and hence not identically zero on the
  connected smooth space $S$.
  The formal inverse of $\Delta(u)$ has coefficients obtained by
  inverting powers of $\Delta(0)$, so these ratios are coefficientwise
  meromorphic. They are independent of the local frame and thus glue.
  
  We first establish regular singularity on a neighborhood of $b$.
  Apply Theorem \ref{thm:hyzz} to the maximal Hodge-invariant F-bundle near
  $\pi_W(b)$. After shrinking, its base is a product
  $\prod_\alpha B_\alpha$, and the F-bundle is the external sum of
  maximal F-bundles on these factors. The selected spectral summand
  is pulled back from a maximal F-bundle of rank $m$ on one factor
  $B_*$, with $\dim B_*=m$.
  
  At the corresponding point $b_*\in B_*$, the residual endomorphism
  is $K^{\mathrm F}|_{\mathcal A}=\lambda\,\id+N$. Its powers span
  the same algebra as the powers of $N$. Cyclicity, therefore, implies that the powers of the Euler vector field of this maximal factor span $T_{B_*,b_*}$. The Euler-generated leaf through $b_*$ is consequently a neighborhood
  of $b_*$ in $B_*$.
  
  The formal-persistence assertion of \cite[Rem.~3.14]{KKPY}, shown in Theorem~\ref{cl:formal-persistence},
  applied to this maximal factor, shows that its fibers remain of
  unramified exponential type near $b_*$ and that their regularized
  formal connections have constant isomorphism class. Pulling back
  along the product projection and choosing the local spectral
  branch through $b$ gives a nonempty analytic neighborhood
  $V\subseteq S^{\mathrm{cyc}}$ on which $\mathcal R_s$ is regular
  singular with exponent multiset $F_b$ for every
  $s\in V(\Bbbk)$.
  
  By Lemma~\ref{lem:scalar-fuchs}, every negative-power coefficient
  of every $a_i$ vanishes at all rational points of $V$.
  These coefficients are analytic on $V$, so the affinoid
  Nullstellensatz, applied locally, makes them zero analytic
  functions. There are only finitely many such coefficients by \eqref{eq:finite-fuchs-bounds}.
  Each extends meromorphically to $S$, and the identity principle
  therefore makes each identically zero. Thus
  \eqref{eq:cyclic-relation} defines a logarithmic lattice at every
  point of $S^{\mathrm{cyc}}(\Bbbk)$.
  
  Put
  \[
  c_i(s):=\text{coefficient of $u^0$ in the Laurent expansion of $a_i(s,u)$},
  \quad
  p_s(T):=T^m+\sum_{i=0}^{m-1}c_i(s)T^i.
  \]
  For $s\in V(\Bbbk)$, Lemma~\ref{lem:scalar-fuchs} identifies
  the roots of $p_s$, counted with multiplicity, with the residue
  eigenvalues of the cyclic logarithmic lattice. If $\alpha$ is
  such a root, then $-\alpha+\Z\in F_b\subset\Q/\Z$.
  Thus $\alpha=q+k$ for some $q\in\Q$ and $k\in\Z$, and in
  particular $\alpha\in\Q$. The coefficients $c_i(s)$, being
  signed elementary symmetric functions of these roots, also
  belong to $\Q$.
  
  Choose an affinoid neighborhood $V'\subseteq V$ of $b$
  on which
  \[
  |c_i(s)-c_i(b)|<1\qquad(0\leq i<m).
  \]
  At every rational point of $V'$, these differences are rational.
  Since the absolute value is trivial on $\Q$, they must vanish.
  The affinoid $V'$ is reduced, and all its rigid points are
  $\Bbbk$-rational. The affinoid Nullstellensatz gives
  $c_i=c_i(b)$ on $V'$. Since $V'$ contains a nonempty analytic
  open neighborhood of $b$, the meromorphic identity principle
  gives $c_i=c_i(b)\quad\text{on }S$. Consequently $p_s=p_b$ on $S^{\mathrm{cyc}}$, and the negative
  roots modulo $\Z$ give $F_s=F_b$.
  
  Finally, $S^{\mathrm{cyc}}$ is dense and open in $S$. Every
  nonempty analytic open subset of $S$ meets $S^{\mathrm{cyc}}$ in a nonempty analytic open subset, which
  contains a $\Bbbk$-rational point.
 \end{proof}
 
 \begin{lemma}[Compare with {\cite[\S3.2]{gong2026irrationalitystabilizationcomplexsmooth}}]\label{lem:separation}
  Let $Y$ be a connected smooth complex projective variety,
  let $Z\subset Y$ be a connected smooth closed subvariety of
  codimension $c\geq2$, and let $E$ be an algebraic vector
  bundle of rank $r\geq1$ on $Y$.
  Consider the comparisons
  \[
  \mathcal F_{\Bl_ZY}
  \simeq
  \mathcal F_Y\boxplus
  \underbrace{\mathcal F_Z\boxplus\cdots\boxplus\mathcal F_Z}_{c-1},
  \qquad
  \mathcal F_{\PP(E)}
  \simeq
  \underbrace{\mathcal F_Y\boxplus\cdots\boxplus\mathcal F_Y}_{r}
  \]
  on the analytic comparison domains of \cite[Thms.~4.5 and 4.11, \S\S5.2.4--5.2.5]{KKPY}.
  In (a) and (b), assume that the initial generic Hodge factor has a cyclic
  representative at a $\Bbbk$-rational point, of unramified exponential
  type with rational exponents.
  \begin{itemize}
   \item[(a)] A generic Hodge factor on the left can be
   represented in a comparison domain and transferred to a generic
   Hodge factor of one summand on the right.
   
   \item[(b)] A generic Hodge factor of any specified summand on
   the right can be represented in a
   comparison domain and transferred to a generic Hodge factor on
   the left.
   
   \item[(c)] In both cases the resulting factor has a cyclic representative
   at a $\Bbbk$-rational point, is of unramified exponential type with rational
   exponents, and has the same rank and regularized exponent multiset as the initial factor.
  \end{itemize}
 \end{lemma}
 \begin{proof}
  Use the common algebraic image of the universal Hodge group
  on the cohomology representations occurring in the comparison.
  The comparison is equivariant with respect to its analytic group germ. The blowup and projective-bundle cohomology decompositions involve Tate twists, which restrict trivially to this group. Taking fixed bases and invariant subbundles, therefore, gives the displayed comparisons of purely even maximal F-bundles.
  
  The proof shall proceed in three steps. First, we describe how the comparison works on a common generic locus. Then, we transfer the selected generic Hodge factor, proving (a) and (b). Finally, we check that the comparison preserves cyclicity, rank, and the regularized exponent multiset.
  
  \begin{itemize}
   \item[(i)] We first describe a common generic locus in a comparison domain.
   Require each external summand to lie in its own generic spectral
   locus, and require the spectra of different summands to be
   disjoint. The first condition holds on a dense open subset:
   the comparison domain is open in the product of the bases, and
   each generic spectral locus is dense and open.
   
   For the second condition, the unit coordinate of the $j$th
   summand changes its operator by $K_j^{\mathrm F}\longmapsto K_j^{\mathrm F}-t_j\id$. The parameters $t_j$ vary independently in sufficiently small
   product neighborhoods. For each pair of summands, the resultant
   of their translated characteristic polynomials is a nonzero
   polynomial in the translation parameters. Its zero locus, therefore, has empty interior. Avoiding the finitely many such
   zero loci gives a dense open subset with pairwise disjoint spectra.
   
   Let $d_j $ be the maximal number of distinct eigenvalues for
   the $j$th summand. On the open subset just constructed, the
   external sum has $\sum_jd_j$ distinct eigenvalues.
   The identified open subset on the left meets its dense global
   generic locus. Hence, the maximal number on the left is also
   $\sum_jd_j$: it equals that number at a point of this
   intersection. Thus, the comparison identifies the reduced
   generic spectral covers on a common nonempty open domain.
   On the external-sum side, the cover is the disjoint union of
   the pulled-back covers of the individual summands.
   
   The component-surjectivity results of
   \cite[(5.15) and \S5.2.5]{KKPY} ensure that
   every selected global generic spectral component has a
   representative in the comparison domains. For a component on
   the external-sum side, the corresponding assertion applies to
   each specified copy of a summand. The restrictions just imposed
   remove only loci with empty interiors, so representatives remain
   on the common generic locus.
   
   \item[(ii)] We now transfer the selected generic Hodge factor. Suppose first that the initial factor is on the left.
   Its cyclic locus is nonempty by hypothesis; hence, it is dense and open.
   Proposition~\ref{prop:cyclic-transport} allows us to choose a
   rational representative on the common generic locus at which
   the factor has its original regularized exponent multiset.
   Since the external spectra are disjoint, the transferred
   factor belongs to one summand.
   
   For the converse, start with the selected component of the
   specified external summand. In the product comparison domain,
   the projection to that summand's base is open. The inverse
   image of its dense cyclic locus therefore meets the common
   generic locus in a nonempty open subset of the corresponding
   spectral component. Choose a rational point there and apply
   Proposition~\ref{prop:cyclic-transport} on the selected summand
   before transferring the factor to the left.
   
   \item[(iii)] Let $\Psi(u)$ be the comparison isomorphism.
   It fixes $u$, is regular and invertible at $u=0$, and intertwines
   the connections. Taking the leading coefficient gives $\Psi(0)K_{\mathrm{left}}^{\mathrm F}=K_{\mathrm{right}}^{\mathrm F}\Psi(0)$. Thus, the selected eigenvalue, rank, and leading Jordan form
   are preserved. Twisting both selected factors by
   $\mathcal E^{\lambda/u}$ gives isomorphic regular singular
   connections, so their exponent multisets agree.
   
   Finally, a single Jordan block on the target summand is cyclic
   with respect to that summand's unit by the regular-representation
   argument preceding \eqref{eq:cyclic-locus}. Hence, the transferred
   factor satisfies all the hypotheses of
   Proposition~\ref{prop:cyclic-transport}.
  \end{itemize}
 \end{proof}
 

 We now classify the possible rank and exponent data in
 dimensions at most two. 
 \begin{proposition}\label{prop:low-dimension}
  Let $W$ be a connected smooth complex projective variety of
  dimension at most two. Suppose that one of its generic Hodge
  factors is cyclic and of unramified exponential type with
  rational residues at a rational point. Its rank $m$ and
  regularized exponent multiset $F$ have one of the following types:
  \begin{enumerate}
   \item \emph{point type}:
   $m=1$ and $F=\ms{\cl0}$;
   \item \emph{curve type}:
   $m=2$ and $F=\ms{\cl{1/2},\cl{1/2}}$;
   \item \emph{surface type}:
   $m=3$ and $F=\ms{\cl0,\cl0,\cl0}$.
  \end{enumerate}
  
  A generic Hodge factor of any projective space that satisfies the same hypotheses is of point type.
 \end{proposition}
 \begin{proof}
  All transfers in this proof use Lemma~\ref{lem:separation};
  in particular, they preserve rank, cyclicity, and the
  regularized exponent multiset.
  
  For a point, $H_{\Hdg}=\Q\mathbf1$, $G=0$, and $K^{\mathrm F}=-y_0\id$.
  Its regularized connection is trivial, so it has point type.
  Since $\PP^a$ is a projective bundle over a point, the
  projective-bundle comparison proves the assertion for every
  projective space.
  
  Let $C$ be a smooth projective curve of genus $g\geq1$.
  Every genus-zero stable map to $C$ is constant. On the
  Hodge-fixed base, the only additive coordinate is the unit
  coordinate, and the quantum product is the cup product.
  Writing $\eta\in H^2(C,\Q)$ for the class with
  $\int_C\eta=1$, we have
  \[
  H_{\Hdg}(C)=\Q\mathbf1\oplus\Q\eta,
  \qquad
  K^{\mathrm F}+y_0\id=(2g-2)\,\eta\cup(-).
  \]
  For $g=1$ this is zero, so the rank-two generalized
  eigenspace is never cyclic. For $g\geq2$ it is a single
  nilpotent Jordan block of size two. Proposition~\ref{prop:nef}
  gives the exponent multiset
  $\ms{\cl{1/2},\cl{1/2}}$. This proves the assertion for curves.
  
  Now let $W$ be a smooth projective surface. Contract its
  $(-1)$-curves to obtain a minimal smooth projective surface
  $\Sigma$. At each step, the selected factor either transfers
  to $\Sigma$ or to a point. In the latter case, it has point type.
  By the classification of minimal smooth projective surfaces
  \cite{BeauvilleSurfaces}, if the minimal surface is rational or ruled, it is either
  $\PP^2$ or a projective bundle over a smooth projective curve.
  The preceding cases and the projective-bundle comparison
  then apply.
  
  In the remaining case, $K_\Sigma$ is nef. By
  Proposition~\ref{prop:nef}, the entire Hodge-invariant fiber
  has a single eigenvalue, and its rank is
  \[
  m=\dim H_{\Hdg}(\Sigma)=\rho(\Sigma)+2.
  \]
  At a cyclic representative, subtract the unit eigenvalue and
  write $N:=K^{\mathrm F}+y_0\id$. The degree argument in
  Proposition~\ref{prop:nef} shows that $N$ strictly raises
  cohomological degree. Since the Hodge-invariant degrees are
  $0,2,4$, we have $N^3=0$. A single Jordan block of size $m$
  therefore has $m\leq3$. Projectivity implies
  $\rho(\Sigma)\geq1$, so $m=3$ and $\rho(\Sigma)=1$.
  
  Proposition~\ref{prop:nef} then gives
  $F=\ms{\cl0,\cl0,\cl0}$.
 \end{proof}
 
 We now prove the rationality criterion directly by weak factorization, using Lemma~\ref{lem:separation} to keep track of the selected cyclic formal factor.
 
 \begin{theorem}\label{thm:criterion}
  Let $X$ be a connected smooth complex projective variety of
  dimension $2\leq n\leq4$. Suppose that a generic Hodge factor
  has a cyclic representative of unramified exponential type
  with rational residues. Let $m$ be its rank and $F$ its
  regularized exponent multiset. If $X$ is rational, then:
  \begin{enumerate}
   \item if $n=2$, then
   \[
   (m,F)=(1,\ms{\cl0});
   \]
   \item if $n=3$, then
   \[
   (m, F)\in
   \left\{
   (1,\ms{\cl0}),
   (2,\ms{\cl{1/2},\cl{1/2}})
   \right\};
   \]
   \item if $n=4$, then
   \[
   (m, F)\in
   \left\{
   (1,\ms{\cl0}),
   (2,\ms{\cl{1/2},\cl{1/2}}),
   (3,\ms{\cl0,\cl0,\cl0})
   \right\}.
   \]
  \end{enumerate}
 \end{theorem}
 \begin{proof}
  Assume that $X$ is rational. By weak factorization
  \cite[Thm.~0.1.1]{AKMW}, a
  birational map $X\dashrightarrow\PP^n$ factors into a finite
  sequence of blowups and blowdowns between smooth projective
  $n$-folds, with smooth centers of codimension at least two.
  
  Track the selected generic Hodge factor along this sequence.
  If the next variety is the blowup of the current one, use
  the converse direction of Lemma~\ref{lem:separation}, with
  the current variety as the specified summand. If the current
  variety is the blowup of the next one, use the forward
  direction of the lemma. The factor then transfers either
  to the next variety or to one of the copies of the center.
  At every step, its rank, cyclicity, unramified exponential
  type, rational residue eigenvalues, and exponent multiset
  are preserved.
  
  The process terminates either with a factor on a center
  of dimension at most $n-2$, or with a factor on $\PP^n$.
  In the latter case Proposition~\ref{prop:low-dimension}
  gives point type. In the former case, the same proposition
  gives point type when $n=2$, point or curve type when
  $n=3$, and point, curve, or surface type when $n=4$.
 \end{proof}

 \section{Applications}
 \label{sec:examples}
 
 Theorem~\ref{thm:criterion} reduces the rationality question to a
 specific task: find a cyclic generic Hodge factor and compute its
 rank and regularized exponent multiset. The rank can be read from
 quantum multiplication, but the exponents require a computation with
 the connection. We explain this computation as a procedure that can
 be checked directly in each example. Its input is the pair of matrices
 $K^{\mathrm F}$ and $G$ in the disk connection
 \[
 \nabla_{\partial_u}=\partial_u+\frac{K^{\mathrm F}}{u^2}+\frac Gu;
 \]
 its output is the pair $(m,F)$ to be compared with the list in
 Theorem~\ref{thm:criterion}.
 
 There are two issues to address. First, we must justify that a factor
 computed at a particular small locus point is one of the generic
 factors to which the theorem applies. Second, we must recover the
 residue exponents after removing that factor's exponential term.
 The following steps address these issues algorithmically.
 
 \paragraph{\textbf{Step 1}:}
 Choose a point on the one-parameter small locus
 \eqref{eq:very-small}, represented by $q\in\Bbbk^\times$ with
 $0<|q|<1$, and evaluate $K^{\mathrm F}$ there. The criterion concerns a \emph{generic} factor, whereas we have
 chosen a specific small locus point. The following determinant
 provides a sufficient test that this point is suitable (Lemma~\ref{lem:cyclic}):
 \[
 \Delta(q):=
 \det\bigl(\mathbf1,K^{\mathrm F}\mathbf1,\ldots,(K^{\mathrm F})^{N-1}\mathbf1\bigr)\neq 0.
 \]
 Indeed, it guarantees that the number of distinct eigenvalues is locally
 constant near the chosen point. The neighborhood meets the dense
 generic spectral locus, so this locally constant number is the
 generic number. Hence, the point belongs to
 $U_X^{\mathrm{gen}}$. Furthermore, $\Delta(q)\ne0$ makes $K^{\mathrm F}$
 a cyclic operator.
 
 \paragraph{\textbf{Step 2}:} Factor the characteristic polynomial of $K^{\mathrm F}$, choose an eigenvalue
 $\lambda$, and let $m$ be its algebraic multiplicity. We seek the
 exponents of the corresponding rank-$m$ factor. For the Fano
 varieties considered here, Theorem~\ref{thm:fano} gives
 \[
 \cM_\lambda\cong\cE^{-\lambda/u}\otimes\cR_\lambda,
 \qquad \operatorname{rk}\cR_\lambda=m,
 \]
 where $\cR_\lambda$ is regular singular. By
 Definition~\ref{def:exponents}, the desired multiset is
 $F=\Exp(\cR_\lambda)$. 
 
Let $y(u)$ denote a horizontal section, meaning
 \[
 y'=\left(-\frac{K^{\mathrm F}}{u^2}-\frac Gu\right)y.
 \]
Recalling Definition~\ref{def:exponents}, a horizontal section of $\cE^{-\lambda/u}$ has scalar factor $e^{\lambda/u}$. Writing $y=e^{\lambda/u}\phi$ and differentiating
 gives:
 \begin{equation}\label{eq:phi-system}
  u^2\phi'=\bigl(\lambda\id-K^{\mathrm F}-uG\bigr)\phi.
 \end{equation}
 
 \paragraph{\textbf{Step 3}:} Seek a formal solution for Equation~\eqref{eq:phi-system} following the ansatz
 \[
 \phi=u^s\sum_{k\geq0}u^kv_k,\qquad v_0\ne0.
 \]
 We call such a solution \emph{log-free}.
 The exponents are rational by Theorem~\ref{thm:fano}. 
\begin{lemma}\label{lem:application-matrix-recursion}
 A log-free solution of \eqref{eq:phi-system} satisfies the recursive relations
 \begin{equation}\label{eq:application-matrix-recursion}
  (\lambda\id-K^{\mathrm F})v_0=0,\qquad
  (\lambda\id-K^{\mathrm F})v_k
  =\bigl(G+(s+k-1)\id\bigr)v_{k-1}\quad(k\geq1).
 \end{equation}
 For every vector $w$ satisfying $w^\top(\lambda\id-K^{\mathrm F})=0$,
 these relations imply the solvability conditions
 \begin{equation}\label{eq:application-left-obstruction}
  w^\top\bigl(G+(s+k-1)\id\bigr)v_{k-1}=0.
 \end{equation}
 Furthermore, every exponent class of $\cR_\lambda$ is represented modulo
 $\Z$ by the leading power $s$ of at least one such solution.
\end{lemma}
 \begin{proof}
Substituting the formal series into \eqref{eq:phi-system} and equating the coefficients of $u^s$ yields the primary relation $(\lambda\id-K^{\mathrm F})v_0=0$. Evaluating the coefficients at order $u^{s+k}$ gives the equality
 $(\lambda\id-K^{\mathrm F})v_k=(G+(s+k-1)\id)v_{k-1}$. Applying the left kernel vector $w^\top$ annihilates the left-hand side, yielding the stated solvability condition.

 Regarding the final assertion, the formal spectral decomposition provided by
 Theorem~\ref{thm:fano}, after twisting by $\cE^{\lambda/u}$, reads
 \[
 \cR_\lambda\ \oplus\!
 \bigoplus_{\nu\ne\lambda}
 \cE^{(\lambda-\nu)/u}\otimes\cR_\nu.
 \]
 Consequently, $\cR_\lambda$ operates as a direct summand of the twisted
 matrix connection: the formal isomorphism maps a vector in
 this initial summand to the original cohomology basis via a
 Laurent-series change of basis. For each residue eigenvalue of $\cR_\lambda$,
 the constant-residue formal normal form of
 \cite[Prop.~3.12]{PutSinger} (valid over the algebraically
 closed characteristic-zero field $\Bbbk$ by
 \cite[Rem.~3.20(1)]{PutSinger}) yields an
 eigenvector solution carrying the opposite power of $u$ entirely free of
 logarithms. Expressing this solution in the original basis through the
 formal meromorphic isomorphism and the inclusion of the direct
 summand produces a solution to \eqref{eq:phi-system}. Since a Laurent-series change of basis modifies the leading power by an integer,
 its class modulo $\Z$ is the same.
 \end{proof}

 \ 

 In practice, the recursive system \eqref{eq:application-matrix-recursion} is solved step by step. Because the selected generic Hodge factor is cyclic, the operator $K^{\mathrm F}$ restricted to this generalized eigenspace consists of a single Jordan block. Consequently, the kernel of $(\lambda\id-K^{\mathrm F})$ is one-dimensional and is spanned by $v_0$. For any $k \geq 1$, once the solvability condition \eqref{eq:application-left-obstruction} is satisfied, the general solution to the inhomogeneous equation takes the form
\begin{equation}\label{eq:kernel-parameter}
    v_k = v_k^{\mathrm{part}} + c_k v_0,
\end{equation}
where $v_k^{\mathrm{part}}$ is a particular solution and $c_k \in \Bbbk$ is an arbitrary scalar constant. In the subsequent computations, we refer to $c_k$ as a \emph{free kernel parameter}.
 
 \subsection{Threefolds}
 
\begin{example}[Cubic threefold]\label{ex:cubic3}
 Let $X\subset\PP^4$ be a smooth cubic threefold and
 $H=c_1(\cO_X(1))$. Lefschetz hyperplane theorem and Poincar\'e duality give
 \[
 H_{\Hdg}(X)=\operatorname{span}_\Q\{\mathbf 1,H,H^2,H^3\}.
 \]
 The primitive middle cohomology has odd degree and contributes
 no rational $(p,p)$ classes. Since $X$ has Fano index two, $K^{\mathrm F}=-2H\star(-)$.
 The operator for small quantum multiplication by $H$ is extracted from
 \cite{Givental}; compare
 \cite[\S3]{cai2026cubicthreefoldsymplecticallyirrational}:
 \begin{equation}\label{eq:cubic-matrix}
  K^{\mathrm F}=-\begin{pmatrix}
   0&12q&0&72q^2\\
   2&0&30q&0\\
   0&2&0&12q\\
   0&0&2&0
  \end{pmatrix},
  \qquad
  G=\operatorname{diag}\Bigl(-\frac32,-\frac12,\frac12,\frac32\Bigr).
 \end{equation}
 The characteristic polynomial and the cyclic determinant evaluate to
 \begin{equation}\label{eq:cubic-spectrum}
  \det(x\id-K^{\mathrm F})=x^2(x^2-108q),
  \qquad
  \det\bigl(\mathbf1,K^{\mathrm F}\mathbf1,(K^{\mathrm F})^2\mathbf1,(K^{\mathrm F})^3\mathbf1\bigr)=64.
 \end{equation}
Thus, the small locus point is generic, and the eigenvalue $\lambda=0$ yields a cyclic generic Hodge factor of rank two.
 We compute the regularized exponents of this factor.
 
 For $\lambda=0$, the right and left kernels of $-K^{\mathrm F}$
 are generated respectively by
 \[
 v_0=(0,-6q,0,1)^\top,\qquad
 w^\top=(1,0,-6q,0).
 \]
 The $k=1$ obstruction in \eqref{eq:application-left-obstruction}
 vanishes identically. Solving the inhomogeneous equation at that order provides
 the particular and general solutions:
 \[
 v_1^{\mathrm{part}}
 =\left(-\frac{q(42s+39)}4,0,\frac{2s+3}4,0\right)^\top,
 \qquad v_1=v_1^{\mathrm{part}}+c v_0.
 \]
 Substituting this general solution into the subsequent obstruction yields an expression strictly independent of the free kernel parameter $c$:
 \[
 w^\top\bigl(G+(s+1)\id\bigr)v_1
 =-\frac{3q}{8}(6s+1)(6s+5).
 \]
 Thus, every exponent class must be represented by one of the
 leading powers $s=-5/6,-1/6$. To verify the two occur, write
 $\phi=(\phi_0,\phi_1,\phi_2,\phi_3)^\top$ and set
 $f=\phi_3$. The last three rows of \eqref{eq:phi-system}
 read
 \[
 \begin{split}
  u(u\partial_u+3/2)f&=2\phi_2,\\
  u(u\partial_u+1/2)\phi_2&=2\phi_1+12qf,\\
  u(u\partial_u-1/2)\phi_1&=2\phi_0+30q\phi_2.
 \end{split}
 \]
 Solving these equations consecutively yields
 \[
 \phi_3=f,\qquad
 \phi_2=\frac u2\Bigl(u\partial_u+\frac32\Bigr)f,\qquad
 \phi_1=\Bigl[\frac{u^2}{4}
 \Bigl(u\partial_u+\frac32\Bigr)^2-6q\Bigr]f,
 \]
 \[
 \phi_0=\Bigl[\frac{u^3}{8}\Bigl(u\partial_u+\frac32\Bigr)^3
 -\frac{qu}{4}(42u\partial_u+39)\Bigr]f.
 \]
 The first row consequently reduces to the scalar differential equation
 \begin{equation}\label{eq:cubic-scalar}
  \Bigl[u^2\Bigl(u\partial_u+\frac32\Bigr)^4
  -3q(6u\partial_u+1)(6u\partial_u+5)\Bigr]f=0.
 \end{equation}
 Conversely, the displayed formulas reconstruct a full horizontal section of
 all four rows from any solution of this scalar equation. Substituting the formal ansatz
 $f=u^s\sum_{k\geq0}a_ku^k$, with $a_0=1$ and
 $a_{-1}=a_{-2}=0$, yields the coefficient recurrence:
 \[
 3q\bigl(6(s+k)+1\bigr)\bigl(6(s+k)+5\bigr)a_k
 =\bigl(s+k-\tfrac12\bigr)^4a_{k-2}.
 \]
 At $k=0$, this confirms the indicial roots $s=-5/6$ and $s=-1/6$. 
 These two admissible leading powers yield two linearly independent log-free solutions, yielding:
  \begin{equation}\label{eq:cubic-exponents}
  m=2,\qquad
  F=\ms{\cl{-5/6},\cl{-1/6}}
  =\ms{\cl{1/6},\cl{5/6}}.
 \end{equation}
 This rank-two factor is not of curve type. Theorem~\ref{thm:criterion} therefore
 implies that every smooth cubic threefold is irrational.
\end{example}
 
 \begin{example}[Ordinary Gushel--Mukai threefold]\label{ex:GM3}
 Let $X\subset\operatorname{Gr}(2,5)$ be the smooth zero locus of a
 section of $\cO(1)^{\oplus2}\oplus\cO(2)$, and let
 $H=c_1(\cO_X(1))$. Adjunction gives $c_1(TX)=H$.
 Using Lefschetz hyperplane theorem and Poincar\'e duality we have
 \[
 H_{\Hdg}(X)=\operatorname{span}_\Q\{\mathbf 1,H,H^2,H^3\}.
 \]
 In the ordered basis $\{\mathbf1,H,H^2,H^3\}$, the operator for small quantum multiplication by $H$ is extracted from \cite[Thm.~6.1.1]{Przyjalkowski}; see also \cite{alexeev2026atomicdecompositioncompleteintersection}:
 \begin{equation}\label{eq:gm-matrix}
  K^{\mathrm F}=-\begin{pmatrix}
   0&156q^2&3600q^3&33120q^4\\
   1&10q&380q^2&3600q^3\\
   0&1&10q&156q^2\\
   0&0&1&0
  \end{pmatrix},
  \qquad
  G=\operatorname{diag}\Bigl(-\frac32,-\frac12,\frac12,\frac32\Bigr).
 \end{equation}
 The characteristic polynomial and the cyclic determinant evaluate to
 \begin{equation}\label{eq:gm-spectrum}
  \det(x\id-K^{\mathrm F})=(x-6q)^2(x^2+32qx-244q^2),
  \qquad
  \det\bigl(\mathbf1,K^{\mathrm F}\mathbf1,(K^{\mathrm F})^2\mathbf1,(K^{\mathrm F})^3\mathbf1\bigr)
  =1.
 \end{equation}
 Since $q\ne0$, the remaining two eigenvalues are distinct from each
 other and from $6q$. The small locus point is therefore generic, and the
 eigenvalue $\lambda=6q$ provides a cyclic generic Hodge factor of rank two.
 
 For $\lambda=6q$, the right and left kernels of
 $6q\id-K^{\mathrm F}$ are generated respectively by
 \[
 v_0=(-360q^3,-60q^2,-6q,1)^\top,
 \qquad
 w^\top=(1,-6q,-60q^2,-360q^3).
 \]
 The initial obstruction $w^\top(G+s\id)v_0$ vanishes
 identically. Solving the inhomogeneous equation at order $k=1$ gives
 \[
 v_1=(-q^2(88s+108),-q(22s+27),s+3/2,0)^\top+c v_0.
 \]
 Substituting this general solution into the subsequent obstruction yields an expression independent of the free kernel parameter $c$:
 \[
 w^\top\bigl(G+(s+1)\id\bigr)v_1
 =-16q^2s(s+1).
 \]
 Because $q\ne0$, every log-free leading power is constrained to $s\in\{0,-1\}$. Lemma~\ref{lem:application-matrix-recursion}
 ensures that such a power represents each exponent class of the regular-singular factor. Both admissible powers belong to the class
 $\cl0$, and the formal factor has rank two. Therefore, the invariants are
 \begin{equation}\label{eq:gm-exponents}
   m=2,\qquad F=\ms{\cl0,\cl0}.
 \end{equation}
 This rank-two factor is not of curve type. Theorem~\ref{thm:criterion} therefore implies the irrationality of every smooth ordinary Gushel--Mukai threefold.
\end{example}
 
 \subsection{Very general quartic fourfold}
 \begin{example}[Very general quartic fourfold]\label{ex:quartic4}
  Let $X\subset\PP^5$ be a smooth quartic fourfold and
  $H=c_1(\cO_X(1))$. Set
  \[
  \mathbf R=\operatorname{span}_\Q\{\mathbf 1,H,H^2,H^3,H^4\}.
  \]
  Assume that:
  \begin{equation}\label{eq:quartic-hodge-condition}
   \rH^4(X,\Q)_{\mathrm{prim}}\cap H^{2,2}(X)=0.
  \end{equation}
  View quartics as hyperplane sections of the fourth Veronese embedding
  of $\PP^5$. Since $h^{3,1}(\PP^5)=0<21=h^{3,1}(X)$,
  \cite[Prop.~3.1]{BFGMP2026} implies this condition for a very general
  quartic fourfold; see also
  \cite[Ex.~6.7]{alexeev2026atomicdecompositioncompleteintersection}. In this case, using
  Lefschetz hyperplane theorem in combination with Poincar\'e duality gives $H_{\Hdg}(X)=\mathbf R$.
  
The matrix of small quantum multiplication is obtained from \cite[Ex.~6.7]{alexeev2026atomicdecompositioncompleteintersection}:
  \begin{equation}\label{eq:quartic-matrix}
   K^{\mathrm F}=-\begin{pmatrix}
    0&48q&0&5568q^2&0\\
    2&0&208q&0&5568q^2\\
    0&2&0&208q&0\\
    0&0&2&0&48q\\
    0&0&0&2&0
   \end{pmatrix},
   \qquad
   G=\operatorname{diag}(-2,-1,0,1,2).
  \end{equation}
  Here $K^{\mathrm F}=-2H\star$, since the Fano index is two. Moreover,
  \begin{equation}\label{eq:quartic-spectrum}
   \det(x\id-K^{\mathrm F})=x^3(x^2-1024q),
   \qquad
   \det\bigl(\mathbf1,K^{\mathrm F}\mathbf1,\ldots,(K^{\mathrm F})^4\mathbf1\bigr)=2^{10}.
  \end{equation}
  Under \eqref{eq:quartic-hodge-condition}, this proves cyclicity
  on the entire Hodge-invariant space. Lemma~\ref{lem:cyclic}
  then shows that the small locus point is generic: the Euler powers
  span the five-dimensional Hodge-fixed tangent space, so their leaf
  contains a neighborhood in which the spectral multiplicities are
  constant, and this neighborhood meets the dense generic locus.
  We select the zero eigenvalue, whose single Jordan block on the
  Hodge-invariant space has size three.
  
  For $\lambda=0$, right and left kernels of $-K^{\mathrm F}$
  are generated by
  \[
  v_0=(-288q^2,0,-24q,0,1)^\top,
  \qquad w^\top=(1,0,-24q,0,-288q^2).
  \]
  The obstructions at $k=1$ and $k=2$ vanish identically. Setting
  the free kernel parameters to zero gives particular solutions
  \[
  \begin{split}
   v_1^{\mathrm{part}}
   &=(0,-q(64s+104),0,s/2+1,0)^\top,\\
   v_2^{\mathrm{part}}
   &=(-2q(29s^2+78s+52),0,s^2/4+s+1,0,0)^\top.
  \end{split}
  \]
  More precisely, the general solutions at these orders are
  \[
  \begin{split}
   v_1&=v_1^{\mathrm{part}}+cv_0,\\
   v_2&=v_2^{\mathrm{part}}
   +c(0,-8q(8s+21),0,(s+3)/2,0)^\top+dv_0.
  \end{split}
  \]
  Keeping both parameters in the $k=3$ solvability condition,
  they cancel and leave
  \[
  w^\top\bigl(G+(s+2)\id\bigr)v_2
  =-16q(s+1)(2s+1)(2s+3).
  \]
  This restricts the possible leading powers to
  $-3/2,-1,-1/2$. 
   They reconstruct three independent vectors, equal to the rank of the selected factor. Therefore
  \begin{equation}\label{eq:quartic-exponents}
   m=3,\qquad
   F=\ms{\cl{-3/2},\cl{-1},\cl{-1/2}}
   =\ms{\cl0,\cl{1/2},\cl{1/2}}.
  \end{equation}
  This rank-three factor is not of surface type, so Theorem~\ref{thm:criterion} implies
  irrationality. In particular, every smooth quartic fourfold satisfying \eqref{eq:quartic-hodge-condition} is irrational, and hence every very general smooth quartic fourfold is irrational.
 \end{example}

 \subsection{Products with the projective line}

 We conclude by showing that smooth ordinary Gushel--Mukai threefolds and smooth cubic threefolds are irrational after taking products with $\mathbb P^1$. We need the following:
 
 \begin{proposition}\label{prop:product}
  Let $V$ be a connected smooth complex projective threefold.
  Suppose it has a cyclic generic Hodge factor of rank two, of
  unramified exponential type with rational residues, whose exponent
  multiset $F$ is different from
  $\ms{\cl{1/2},\cl{1/2}}$. Then $V\times\PP^1$ is irrational.
 \end{proposition}
 \begin{proof}
  Write
  \[
  V\times\PP^1\cong\PP(\cO_V\oplus\cO_V).
  \]
  The projective-bundle comparison of Lemma~\ref{lem:separation}
  identifies the Hodge-invariant F-bundle of this product, on a
  comparison domain, with the external sum of two copies of
  $\mathcal F_V$. The converse direction of the same lemma transfers
  the selected factor of either specified copy to a cyclic generic
  Hodge factor on $V\times\PP^1$, preserving its rank and exponent
  multiset.
  For a rational fourfold, Theorem~\ref{thm:criterion} allows only
  curve type among rank-two cyclic generic Hodge factors. The
  transferred factor violates this condition, proving irrationality.
 \end{proof}
 
 Irrationality of ordinary Gushel--Mukai threefolds and smooth cubic threefolds after products with $\mathbb P^1$ follows: for a smooth cubic threefold $V$, the factor with exponents
 $\ms{\cl{1/6},\cl{5/6}}$ from Example~\ref{ex:cubic3}
 shows that $V\times\PP^1$ is irrational; compare
 \cite{gong2026irrationalitystabilizationcomplexsmooth}.
 Likewise, the factor with exponents $\ms{\cl0,\cl0}$ from
 Example~\ref{ex:GM3} shows that $V\times\PP^1$ is irrational
 for every smooth ordinary Gushel--Mukai threefold $V$.
 
 \medskip
 
 \appendix
 
  \section{Formal persistence along the powers of the Euler field}
 \label{app:formal-persistence}
 
 We keep using the notations introduced in Sections~\ref{sec:F-bundles} and
 \ref{sec:formal}. Namely, given an F-bundle $(\mathcal H,\nabla)/B$, we set $\mathcal H_0:=\cH|_{u=0}$; $\mu$ denotes the action on $\cH_{0}$, $K^{\mathrm F}$ is the residual endomorphism, and $\widehat\cH$ and $\widehat\cH^{\mathrm{mer}}$ are the bundles over $B$ with fibers $\widehat\cH_b=(\cH|_{\{b\}\times\mathbb D})_0\otimes_{\cO_{\mathbb D,0}}\Bbbk[[u]]$ and $\widehat\cH_b^{\mathrm{mer}}=\widehat\cH_b[u^{-1}]$, respectively.
 
    For a maximal F-bundle, $\circ$, $e$, and $\Eu^{\mathrm F}$ denote the tangent product, its unit, and the F-bundle Euler field. Here we prove Theorem~\ref{cl:formal-persistence}, which is the formal-persistence statement of \cite[Rem.~3.14]{KKPY}. 
 
 \begin{lemma}\label{lem:app-transport-operators}
  Let $(\cH,\nabla)/B$ be an F-bundle, let $p\in\Bbbk[z]$,
  and let $X_p$ be an even analytic vector field on $B$ satisfying
  $\mu(X_p)=p(K^{\mathrm F})$. On $\widehat\cH^{\mathrm{mer}}$, set
  \begin{equation}\label{eq:app-transport-operators}
   x:=\nabla_{u^2\partial_u},\qquad
   \partial_x:=u^{-1},\qquad
   \widehat\nabla_p:=\nabla_{X_p}-p(x)\partial_x.
  \end{equation}
  Here $u$ and $\partial_x$ act by multiplication, and products of operators denote composition, with the rightmost operator
  acting first. Then $x$ and $\widehat\nabla_p$ preserve every
  shifted lattice $u^a\widehat\cH$, $a\in\Z$. Moreover, $x_{|\widehat\cH}\bmod u=K^{\mathrm F}$, and
  \begin{equation}\label{eq:app-transport-commutators}
   [\widehat\nabla_p,x]=-p(x),\qquad
   [\widehat\nabla_p,\partial_x]=p'(x)\partial_x,\qquad
   [\widehat\nabla_p,u]=-up'(x).
  \end{equation}
 \end{lemma}
 \begin{proof}
  By the defining pole condition of an F-bundle, $x$ preserves
  $\widehat\cH$ and reduces to $K^{\mathrm F}$ modulo $u$.
  The Leibniz rule gives, for every $a\in\Z$,
  \begin{equation}\label{eq:app-x-shift}
   x(u^av)=u^a x(v)+a u^{a+1}v,
   \qquad v\in\widehat\cH^{\mathrm{mer}}.
  \end{equation}
  Thus $x$, and hence every polynomial in $x$, preserves all
  the shifted lattices. Taking $a=-1$ in
  \eqref{eq:app-x-shift} yields $[\partial_x,x]=\id$. Consequently,
  \begin{equation}\label{eq:app-polynomial-commutators}
   [\partial_x,p(x)]=p'(x),\qquad
   [p(x),u]=up'(x)u.
  \end{equation}
  The second identity follows from the first by commuting
  $p(x)$ with $u\partial_x=\id$.
  
  We next check that $\widehat\nabla_p$ preserves
  $\widehat\cH$. Both $u\nabla_{X_p}$ and $p(x)$ preserve
  this lattice, and their reductions are equal:
  \[
  (u\nabla_{X_p})\bmod u=\mu(X_p)
  =p(K^{\mathrm F})=p(x)\bmod u.
  \]
  Their difference, therefore, has image in $u\widehat\cH$.
  Moving $\partial_x=u^{-1}$ to the left in
  \eqref{eq:app-transport-operators}, using
  \eqref{eq:app-polynomial-commutators}, gives
  \begin{equation}\label{eq:app-lattice-transport}
   \widehat\nabla_p
   =u^{-1}\bigl(u\nabla_{X_p}-p(x)\bigr)+p'(x).
  \end{equation}
  Each term on the right preserves $\widehat\cH$.
  
  Since $X_p$ is a vector field on $B$, it is independent of
  $u$. Flatness and the Leibniz rule give
  $[\nabla_{X_p},x]=[\nabla_{X_p},\partial_x]=0$.
  Using \eqref{eq:app-polynomial-commutators}, we obtain
  \[
  \begin{aligned}
  [\widehat\nabla_p,x]
  &=-[p(x)\partial_x,x]=-p(x),\\
  [\widehat\nabla_p,\partial_x]
  &=-[p(x),\partial_x]\partial_x=p'(x)\partial_x.
  \end{aligned}
  \]
  Commuting $\widehat\nabla_p$ with $u\partial_x=\id$
  then gives $[\widehat\nabla_p,u]=-up'(x)$, proving
  \eqref{eq:app-transport-commutators}.
  Finally, these identities can be written as
  \[
  \widehat\nabla_p u=u\bigl(\widehat\nabla_p-p'(x)\bigr),
  \qquad
  \widehat\nabla_p u^{-1}
  =u^{-1}\widehat\nabla_p+p'(x)u^{-1}.
  \]
  Starting with $\widehat\cH$ and using that $p'(x)$
  preserves every shifted lattice, induction in both directions
  proves the assertion for $u^a\widehat\cH$.
 \end{proof}
 
 For a maximal F-bundle over a purely even base, set
 \begin{equation}\label{eq:app-polynomial-fields}
  X_p:=p_\circ(\Eu^{\mathrm F}),\qquad p\in\Bbbk[z],
  \qquad (\Eu^{\mathrm F})^{\circ0}:=e.
 \end{equation}

 \begin{lemma}\label{lem:key}
  Let $(\cH,\nabla)/B$ be a maximal F-bundle over a purely
  even base. For $p,q\in\Bbbk[z]$, the fields defined in
  \eqref{eq:app-polynomial-fields} satisfy
  \begin{equation}\label{eq:key}
   \mu([X_p,X_q])=(p'q-pq')(K^{\mathrm F}).
  \end{equation}
  Equivalently,
  \begin{equation}\label{eq:involution}
   [X_p,X_q]=X_{p'q-pq'}.
  \end{equation}
 \end{lemma}
 \begin{proof}
  Rearranging \eqref{eq:app-lattice-transport} gives
  \[
  u\nabla_{X_p}=p(x)+u\bigl(\widehat\nabla_p-p'(x)\bigr).
  \]
  All operators in parentheses preserve the $u$-adic filtration.
  Moreover, $[p(x),u]$ has image in $u^2\widehat\cH$ by
  \eqref{eq:app-polynomial-commutators}. Expanding the
  commutator and discarding terms with image in
  $u^2\widehat\cH$ therefore gives
  \begin{equation}\label{eq:app-bracket-reduction}
  \begin{aligned}
  [u\nabla_{X_p},u\nabla_{X_q}]
  &\equiv u\bigl([p(x),\widehat\nabla_q]
  +[\widehat\nabla_p,q(x)]\bigr)\\
  &=u\bigl(p'(x)q(x)-p(x)q'(x)\bigr)
  \pmod{u^2\widehat\cH}.
  \end{aligned}
  \end{equation}
  For the equality, use the first commutator in
  \eqref{eq:app-transport-commutators} and  Leibniz rule. On the other hand, flatness gives
  \[
  [u\nabla_{X_p},u\nabla_{X_q}]=u^2\nabla_{[X_p,X_q]}.
  \]
  Dividing \eqref{eq:app-bracket-reduction} by $u$ and reducing
  modulo $u$ proves \eqref{eq:key}. The injectivity of $\mu$ and Equation~\eqref{eq:F-Euler-powers} give \eqref{eq:involution}.
 \end{proof}
 
 \begin{theorem}[{\cite[Rem.~3.14]{KKPY}}]
  \label{cl:formal-persistence}
  Let $(\cH,\nabla)/B$ be a maximal F-bundle over a smooth,
  purely even base, and let $b\in B(\Bbbk)$. Use the notation
  \eqref{eq:app-polynomial-fields}, and put
  \begin{equation}\label{eq:app-euler-distribution}
   \mathcal D_w:=\operatorname{span}_{\Bbbk}
   \{X_p(w):p\in\Bbbk[z]\}
   =\operatorname{span}_{\Bbbk}
   \{(\Eu_w^{\mathrm F})^{\circ j}:j\geq0\}.
  \end{equation}
  Then the following statements hold.
  \begin{enumerate}
   \item There is a unique germ of a closed smooth analytic
   submanifold $W\subset B$ through $b$ such that
   $T_wW=\mathcal D_w$ for all sufficiently nearby rational
   points $w\in W(\Bbbk)$.
   \item The number of distinct eigenvalues of
   $\Eu^{\mathrm F}\circ(-)|_{T_W}$ and their algebraic
   multiplicities are locally constant on $W$.
   \item Suppose that $\widehat\cH_b^{\mathrm{mer}}$ is of
   unramified exponential type. After shrinking $W$, the
   distinct eigenvalues of $K^{\mathrm F}|_W$ admit analytic
   labels $\lambda_\alpha(w)$, and there are regular-singular
   connections $\cR_\alpha$ over $\Bbbk((u))$ such that
   \begin{equation}\label{eq:app-constant-factors}
    \widehat\cH_w^{\mathrm{mer}}
    \cong\bigoplus_\alpha
    \cE^{-\lambda_\alpha(w)/u}\otimes\cR_\alpha,
    \qquad w\in W(\Bbbk).
   \end{equation}
   Here, the $\cR_\alpha$ are the regular-singular factors at
   $b$ and do not depend on $w$, up to isomorphism.
  \end{enumerate}
 \end{theorem}
 \begin{proof}
  We first construct $W$ and prove the spectral assertion.
  We then construct parallel transport for
  $\widehat\nabla_p$ along a flow of $X_p$ and show that it
  preserves the regular-singular factors.
  
  \medskip
  \noindent\emph{Construction of the integral germ.}
  Let $n=\dim B$. By the Cayley--Hamilton theorem and
  \eqref{eq:F-Euler-powers}, the fields
  $X_1,X_z,\ldots,X_{z^{n-1}}$ generate the distribution
  \eqref{eq:app-euler-distribution} over $\cO_B$.
  Equation~\eqref{eq:involution} shows that it is involutive.
  Its rank need not be constant in a neighborhood in $B$, so
  we construct its integral germ using flows.
  
  The local analytic flows of the fields $X_p$ exist by
  \cite[Lem.~3.6]{HYZZ}. For each such flow,
  \eqref{eq:involution}, followed by the Cayley--Hamilton theorem,
  gives a finite linear differential system for the pullbacks
  of the generators. Thus, the flow preserves $\mathcal D$ and
  its rank. Choose $r=\dim\mathcal D_b$ of the generators
  whose values at $b$ are independent. Their successive flows
  define an analytic map from an $r$-dimensional polydisk to
  $B$, with injective differential at the origin. After
  shrinking, its image is a closed embedded smooth germ $W$.
  Flow invariance gives $\dim\mathcal D_w=r$ on this image,
  and the differential of the parametrization has image
  $\mathcal D_w$. Hence $T_wW=\mathcal D_w$.
  Uniqueness follows because any other integral germ through $b$ has dimension $r$ and is preserved by the same flows, thus agreeing as a germ. This proves (i).
  
  \medskip
  \noindent\emph{Spectral constancy on $T_W$.}
  Write $L:=\Eu^{\mathrm F}\circ(-)|_{T_W}$; it preserves
  $T_W$ because $LX_q=X_{zq}$. Equation~\eqref{eq:involution}
  gives
  \[
  (\mathrm{Lie}_{X_p}L)(X_q)
  =[X_p,X_{zq}]-L[X_p,X_q]=-X_{pq}.
  \]
  Since the fields $X_q$ span $T_W$, this is the identity
  \begin{equation}\label{eq:Lie-derivative-equality}
   \mathrm{Lie}_{X_p}L=-p(L).
  \end{equation}
  Let $\Phi_s$ be the local flow of $X_p$, and let
  $\varphi_s$ be the scalar analytic flow determined by
  \begin{equation}\label{eq:app-scalar-flow}
   \partial_s\varphi_s(z)=-p(\varphi_s(z)),\qquad
   \varphi_0(z)=z.
  \end{equation}
  Thus \eqref{eq:Lie-derivative-equality} gives
  $\Phi_s^*L=\varphi_s(L)$. For small $s$, the function
  $\varphi_s$ is injective on the finite set $\spec(L_b)$
  and $\varphi_s'(\lambda)\ne0$ at each of its elements.
  On a generalized $\lambda$-eigenspace, write
  $L_b=\lambda\id+N$, where $N$ is nilpotent. Taylor expansion
  gives
  \begin{equation}\label{eq:app-jordan-transport}
   \varphi_s(\lambda\id+N)-\varphi_s(\lambda)\id
   =N g_s(N),\qquad g_s(0)=\varphi_s'(\lambda)\ne0.
  \end{equation}
  The operator $g_s(N)$ is invertible and commutes with $N$.
  Consequently, \eqref{eq:app-jordan-transport} has the same
  kernels of powers as $N$. The eigenvalue multiplicities, and
  indeed the Jordan block sizes, are therefore unchanged along
  each generating flow. Composing these flows proves (ii).
  
  \medskip
  \noindent\emph{Parallel transport on the completed lattices.}
  Fix a polynomial $p$ and a flow curve $b_s=\Phi_s(b)$.
  We use the operators $x$, $\partial_x$, and
  $\widehat\nabla_p$ of \eqref{eq:app-transport-operators}.
  In an analytic trivialization along this curve, write
  \begin{equation}\label{eq:app-analytic-operators}
   x=u^2\partial_u+C(s,u),\qquad
   \nabla_{X_p}=\partial_s+u^{-1}B(s,u),\qquad
   B(s,0)=p(C(s,0)),
  \end{equation}
  where $B$ and $C$ are analytic matrices. Put
  $\delta:=u\partial_u$, so that $x=u\delta+C(s,u)$.
  When $p(u\delta+C(s,u))$ is written with all powers of
  $\delta$ on the right, its constant term modulo $u$ is
  $p(C(s,0))$, and all its positive-degree coefficients are
  divisible by $u$. Equation~\eqref{eq:app-lattice-transport}
  therefore takes the form
  \begin{equation}\label{eq:app-bounded-transport}
   \widehat\nabla_p=\partial_s+A(s,u,\delta),
  \end{equation}
  with $A$ polynomial in $\delta$ and analytic in $(s,u)$.
  
  On a sufficiently small closed $u$-disk, $\delta$ is bounded
  in the Gauss norm, with $\|\delta\|\leq1$. Thus $A$ defines
  an analytic family of bounded operators. Write
  $A(s)=\sum_{j\geq0}A_js^j$, and choose $M\geq1$ such that
  $\|A_j\|\leq M^{j+1}$. Parallel transport is obtained by
  solving
  \begin{equation}\label{eq:app-transport-recursion}
  \begin{gathered}
  (\partial_s+A(s))T_s=0,\qquad T_0=\id,\\
  T_s=\sum_{n\geq0}T_ns^n,\qquad
  (n+1)T_{n+1}=-\sum_{j=0}^n A_jT_{n-j}.
  \end{gathered}
  \end{equation}
  The absolute value on $\Bbbk$ is trivial on $\Q$, so
  $|n+1|=1$. The recursion gives $\|T_n\|\leq M^n$;
  hence $T_s$ converges in operator norm for $|s|<M^{-1}$.
  The equation $\partial_sS_s=S_sA(s)$, $S_0=\id$, gives
  the same bound. Differentiating $S_sT_s$ and $T_sS_s$
  shows that $S_s=T_s^{-1}$.
  
  To act on arbitrary formal series, perform the same
  construction on $\widehat\cH_b/u^N\widehat\cH_b$.
  The bounds are independent of $N$: the operators preserve
  the $u$-adic filtration, and the quotient Gauss norms satisfy
  the same estimates. Taking inverse limits therefore defines
  mutually inverse maps on the entire completed lattices,
  including formal series that converge on no $u$-disk.
  For a shifted lattice $u^a\widehat\cH_b$, conjugation by
  $u^a$ replaces $\delta$ in
  \eqref{eq:app-bounded-transport} by $\delta+a$.
  Since $|a|\leq1$, the bounds remain uniform in $a\in\Z$.
  Uniqueness in \eqref{eq:app-transport-recursion} makes these
  constructions compatible. We obtain $\Bbbk$-linear
  isomorphisms
  \begin{equation}\label{eq:app-transport-maps}
   T_s\colon\widehat\cH_b^{\mathrm{mer}}
   \longrightarrow\widehat\cH_{b_s}^{\mathrm{mer}},\qquad
   T_s(u^a\widehat\cH_b)=u^a\widehat\cH_{b_s}
   \quad(a\in\Z).
  \end{equation}
  These maps need not be $\Bbbk((u))$-linear.
  
  \medskip
  \noindent\emph{Evolution of the operators and spectral summands.}
  For an operator $O$ along the flow, differentiating its
  conjugate gives
  \[
  \partial_s(T_s^{-1}O_{b_s}T_s)
  =T_s^{-1}[\widehat\nabla_p,O]_{b_s}T_s.
  \]
  Apply this identity to $x$, $\partial_x$, and $u$.
  The commutators \eqref{eq:app-transport-commutators} and
  the initial conditions at $s=0$ yield
  \begin{equation}\label{eq:app-integrated-transport}
  \begin{aligned}
  x_s&:=T_s^{-1}x_{b_s}T_s=\varphi_s(x_b),\\
  \partial_{x,s}&:=T_s^{-1}\partial_{x,b_s}T_s
  =\varphi_s'(x_b)^{-1}\partial_{x,b},\\
  u_s&:=T_s^{-1}uT_s=u\varphi_s'(x_b).
  \end{aligned}
  \end{equation}
  Here $\varphi_s$ is the scalar flow in
  \eqref{eq:app-scalar-flow}. Its functions of $x_b$ are
  defined by convergent power series in $s$ on each truncated
  lattice, followed by passage to the inverse limit.
  Notice that $u_s$ denotes the transported multiplication
  operator, rather than a new scalar coordinate on the original
  fiber. 
  
  Reduction of the first identity modulo $u$ conjugates
  $K_{b_s}^{\mathrm F}$ to $\varphi_s(K_b^{\mathrm F})$.
  Thus the distinct eigenvalues of $K^{\mathrm F}$ evolve as
  $\lambda_s=\varphi_s(\lambda)$, with their multiplicities
  unchanged. Along the successive flows parametrizing $W$,
  this gives analytic labels $\lambda_\alpha(w)$.
  
  By Lemma~\ref{lem:splitting}, write
  $\widehat\cH_b=\bigoplus_\lambda\cL_\lambda$ for the
  $x_b$-stable spectral decomposition. Each $\cL_\lambda$
  is preserved by $\varphi_s(x_b)$ and
  $\varphi_s'(x_b)^{\pm1}$ for small $s$. Consequently,
  \eqref{eq:app-integrated-transport} gives
  \begin{equation}\label{eq:app-transported-lattice}
   u_s\cL_\lambda=u\cL_\lambda,\qquad
   u\,T_s(\cL_\lambda)=T_s(u\cL_\lambda).
  \end{equation}
  The image $T_s(\cL_\lambda)$ is therefore a complete
  $\Bbbk[[u]]$-submodule, stable under $x_{b_s}$.
  Lifting a basis modulo $u$ gives a finite free basis of this
  image; its reduction is the generalized
  $\varphi_s(\lambda)$-eigenspace of $K_{b_s}^{\mathrm F}$.
  These images form a direct sum. The uniqueness in
  Lemma~\ref{lem:splitting} identifies them with the spectral
  lattices at $b_s$. Localizing
  \eqref{eq:app-transported-lattice} also identifies the
  corresponding meromorphic summands under $T_s$.
  
  \medskip
  \noindent\emph{Preservation of the regular-singular factors.}
  Assume now that $\widehat\cH_b^{\mathrm{mer}}$ is of
  unramified exponential type. Fix a spectral summand with
  eigenvalue $\lambda$. By Corollary~\ref{cor:leading-exponential},
  it is isomorphic to $\cE^{-\lambda/u}\otimes\cR$, with
  $\cR$ regular singular. Choose a constant-residue basis for
  $\cR$ as in \cite[Prop.~3.12 and Rem.~3.20(1)]{PutSinger}.
  Write $v$ for the tuple of induced basis elements of the
  exponential twist and $R$ for the constant residue matrix, so
  \begin{equation}\label{eq:app-initial-normal-form}
   \nabla_{\partial_u}v
   =\left(\frac{\lambda}{u^2}+\frac Ru\right)v,
   \qquad x_bv=(\lambda+uR)v.
  \end{equation}
  On this summand abbreviate $x_b$ and $\partial_{x,b}$ to
  $x$ and $\partial_x$, and put $z:=x-\lambda$.
  Then $z=u(u\partial_u+R)$ in this basis. In particular,
  $z$ strictly raises the $u$-adic order on the
  constant-residue lattice $\cL:=\Bbbk[[u]]v$. Since
  $[\partial_x,z]=\id$, \eqref{eq:app-initial-normal-form}
  implies
  \begin{equation}\label{eq:app-centered-residue}
   z\partial_xv=(\partial_xz-\id)v=(R-\id)v.
  \end{equation}
  
  Set $\lambda_s:=\varphi_s(\lambda)$ and factor the
  centered scalar flow as
  \begin{equation}\label{eq:app-centered-flow}
   \psi_s(z):=\varphi_s(\lambda+z)-\varphi_s(\lambda)
   =a_s z h_s(z),\qquad
   a_s:=\varphi_s'(\lambda)\ne0,\qquad h_s(0)=1.
  \end{equation}
  In the notation of \eqref{eq:app-integrated-transport},
  \begin{equation}\label{eq:app-centered-operators}
   z_s:=x_s-\lambda_s=\psi_s(z),\qquad
   \partial_{x,s}=\psi_s'(z)^{-1}\partial_x,\qquad
   u_s=u\psi_s'(z).
  \end{equation}
  We will find a basis for the scalar action defined by $u_s$
  in which $x_s$ has the same residue matrix $R$.
  
  Because $R$ commutes with $z$ and $z$ strictly raises
  the filtration on $\cL$, the following formal operator and
  basis tuple are well-defined:
  \begin{equation}\label{eq:app-corrected-basis}
   H_s(z):=h_s(z)^{R-\id}
   =\exp\bigl((R-\id)\log h_s(z)\bigr),\qquad
   w:=H_s(z)v.
  \end{equation}
  The commutator $[\partial_x,z]=\id$ gives
  $[\partial_x,H_s(z)]=H_s'(z)$, with
  \[
  H_s'(z)=(R-\id)H_s(z)\frac{h_s'(z)}{h_s(z)}.
  \]
  Combining this with \eqref{eq:app-centered-residue} and
  \eqref{eq:app-centered-flow}, we obtain
  \begin{equation}\label{eq:app-residue-calculation}
  \begin{aligned}
  z_s\partial_{x,s}w
  &=\left(\id+z\frac{h_s'(z)}{h_s(z)}\right)^{-1}
    z\partial_x\bigl(H_s(z)v\bigr)\\
  &=\left(\id+z\frac{h_s'(z)}{h_s(z)}\right)^{-1}
    \bigl(zH_s'(z)v+H_s(z)(R-\id)v\bigr)\\
  &=(R-\id)w.
  \end{aligned}
  \end{equation}
  Also $[\partial_{x,s},z_s]=\id$ and
  $u_s\partial_{x,s}=\id$, by
  \eqref{eq:app-centered-operators}. Hence
  \eqref{eq:app-residue-calculation} gives
  \begin{equation}\label{eq:app-transported-normal-form}
   z_sw=u_s(\partial_{x,s}z_s)w=u_sRw,
   \qquad x_sw=(\lambda_s+u_sR)w.
  \end{equation}
  
  It remains to check that this operator calculation produces
  an actual $\Bbbk((u))$-basis at $b_s$. By
  \eqref{eq:app-corrected-basis}, $w\equiv v\pmod{u\cL}$,
  while \eqref{eq:app-centered-operators} gives, for $k\geq0$,
  \[
  u_s(u^k\cL)=u^{k+1}\cL,\qquad
  u_s=a_su\pmod{\text{operators raising the $u$-adic order by two}}.
  \]
  Successive reduction modulo powers of $u$ therefore expresses
  each element of $\cL$ uniquely as a formal series in $u_s$
  with coefficients in the $\Bbbk$-span of the entries of $w$.
  The lattice $\cL$ is commensurable with the original spectral
  lattice, and \eqref{eq:app-transport-maps} ensures that $T_s$
  is continuous for their $u$-adic topologies. Since
  $T_su_s=uT_s$, the tuple $T_sw$ is a
  $\Bbbk[[u]]$-basis of $T_s(\cL)$ and a
  $\Bbbk((u))$-basis of the transported summand.
  Applying $T_s$ to \eqref{eq:app-transported-normal-form}
  and dividing by $u^2$ gives
  \begin{equation}\label{eq:app-final-normal-form}
   \nabla_{\partial_u}(T_sw)
   =\left(\frac{\lambda_s}{u^2}+\frac Ru\right)(T_sw).
  \end{equation}
  Thus \eqref{eq:app-final-normal-form} identifies the summand
  at $b_s$ with $\cE^{-\lambda_s/u}\otimes\cR$: its
  regular-singular factor has the same constant residue matrix
  as in \eqref{eq:app-initial-normal-form}.
  
  Apply this argument to every initial spectral summand and
  then to the successive generating flows parametrizing $W$.
  After shrinking their parameter polydisk, these identifications
  hold throughout $W$ and give \eqref{eq:app-constant-factors}.
  This proves (iii).
 \end{proof}
 
 \begin{remark}\label{rem:compatible-persistence}
  The formal-persistence conclusion also holds for an additional
  F-bundle $(\widetilde\cH,\widetilde\nabla)$ on the same leaf
  $W$, not necessarily maximal, provided that its action and
  residual endomorphism satisfies
  \begin{equation}\label{eq:app-compatible-action}
   \widetilde\mu(X_p)=p(\widetilde K^{\mathrm F}),
   \qquad p\in\Bbbk[z],\qquad
   X_p=p_\circ(\Eu^{\mathrm F}).
  \end{equation}
  Here $\Eu^{\mathrm F}$ is the Euler field of the original
  maximal F-bundle, and $\widetilde\mu$ and
  $\widetilde K^{\mathrm F}$ are defined by
  \eqref{eq:F-residual-action} for the additional bundle.
  If $\widehat{\widetilde\cH}_b^{\mathrm{mer}}$ is of
  unramified exponential type, its regular-singular factors are
  locally constant up to isomorphism along $W$, with eigenvalues
  matched by the flows above.
  
  Indeed, maximality was used to construct the tangent algebra
  and the leaf. Condition~\eqref{eq:app-compatible-action}
  allows us to apply Lemma~\ref{lem:app-transport-operators}
  to the additional bundle. The parallel transport and
  constant-residue calculation then apply without change.
  If the additional bundle is a super vector bundle, all these
  operators are even, and the argument applies to each parity
  separately.
 \end{remark}

 \printbibliography
 
 \medskip
\end{document}